\documentclass[12pt]{amsart}
\usepackage[english]{babel}
\usepackage[utf8]{inputenc}
\usepackage{times}
\usepackage[T1]{fontenc}
\usepackage{graphicx}
\usepackage{amscd,amsmath,amsfonts,amstext,amssymb,amsbsy,amsopn,amsthm,eucal}
\usepackage{txfonts}%times math fonts. GOES AFTER amsmath PACKAGE, OTHERWISE LATEX GIVES THE ERROR command ''\iint`` already defined.
\usepackage[figuresright]{rotating}
\usepackage{rotating,booktabs}
\usepackage{caption}
\usepackage{multirow}
\usepackage[table]{xcolor}
\usepackage{fancyhdr}
\usepackage{color} %PACCHETTO PER SCRIVERE IL TESTO COLORATO. si usa \textcolor{colore o codice}{testo}. OKKIO: con questo pacchetto le linee sopra e sotto la pagina diventano blu SOLO SE GUARDATE DAL DVI, nel pdf sono normali (BOH...)
\usepackage{hyperref}
\hypersetup{
  pdftitle={Warped product},
  pdfauthor={CLW},
  pdfsubject={WP},
  pdfkeywords={hypersurface, Kottler-Schwartzchild space,$\sigma_k$ equation, $C^2$ estimates}
  pdfpagelayout=SinglePage,
  pdfpagemode=UseOutlines,  pdfpagemode=UseOutlines,
  colorlinks,
  linkcolor=[rgb]{0,0,0.7},
  urlcolor=[rgb]{0,0,0.4},
  citecolor=[rgb]{0.4,0.1,0}
  }

\newcommand{\abc}[1]{\left( #1 \right)}%()è?è?ü???p?èò?§?%
\newcommand{\abz}[1]{\left[ #1 \right]}%[]è?è?ü???p?èò?§?%

\newcommand{\n}{\nabla}

\newcommand{\be}{\begin{equation}}
\newcommand{\ee}{\end{equation}}

\newcommand{\nei}{\nabla_{E_i}}
\newcommand{\nej}{\nabla_{E_j}}
\newcommand{\nij}{\nabla^2_{E_i,E_j}}
\newcommand{\nii}{\nabla^2_{E_i,E_i}}

\def\dd{\mathrm{d}}

\def\ae{\langle}
\def\ad{\rangle}

\def\dd{\textrm{d}}

\def\dd{\textrm{d}}
\newcommand{\p}{\partial}
\newcommand{\al}{\alpha}

\def\dfrac{\displaystyle\frac}
\def\dsum{\displaystyle\sum}

\newtheorem{prop}{Proposition}[section]
\newtheorem{thm}[prop]{Theorem}

\newtheorem{lem}[prop]{Lemma}

\newtheorem{rem}[prop]{Remark}

\newtheorem{defn}[prop]{Definition}

\numberwithin{equation}{section}

\begin{document}

%%%%% To ease editing, add:

\baselineskip=17pt

%%%%%%%%%%%%%%%%

\title[Prescirbed Weingarden curvature]{Starshaped compact hypersurfaces with prescirbed Weingarden curvature in warped product manifolds}

\author[Chen, Li and Wang]{Daguang Chen, Haizhong Li and Zhizhang Wang}

\thanks{*The first author was supported by  NSFC grant 11471180, the second author was supported by NSFC grant No.11671224, and the third
author was partially supported  NSFC grant No.11301087 and No. 11671069.}

\keywords{Weingarden curvature, $\sigma_k$, $C^2$ estimates, hypersurfaces}
\subjclass[2010]{Primary 53C45, Secondary 53J60}

\maketitle

%\tableofcontents

%-------------------------------------------------------------------------
\begin{abstract}
Given a compact Riemannian manifold $M$, we consider a
warped product  $\bar M = I \times_h M$  where $I$ is an open
interval in $\Bbb R$. For a positive function $\psi$ defined on $\bar M$, we generalized  the arguments in \cite{GRW2015} and \cite{RW16}, to obtain the curvature estimates for Hessian equations $\sigma_k(\kappa)=\psi(V,\nu(V))$.
We also obtain some existence results for the starshaped compact hypersurface $\Sigma$ satisfying the above  equation with various assumptions.
\end{abstract}

\section{Introduction}
Assume that $\Sigma^n$  is a hypersurface in Riemannian manfold $\bar{M}^{n+1}$.
The Weingarten curvature equation is given by
\begin{eqnarray*}
\sigma_k(\kappa(X))=\psi(X), \ \ \forall\,  X\in \Sigma,
\end{eqnarray*}
where $X$ is the position vector field of hypersurfce $\Sigma$ in $\bar{M}^{n+1}$ and $\sigma_k$ is the $k^{th}$ elementary symmetric function.

Finding closed hypersurfaces with prescribed Weingarten curvature in Riemannian manifolds attracts many authors’ interest. Such results were obtained
for case of prescribing mean curvature by Bakelman-Kantor \cite{BK73,BK74}  and by Treibergs-Wei \cite{TW} in the Euclidean space, for the case of prescribing Gaussian curvature by Oliker \cite{Oliker84}, and for general Weingarten curvatures by Aleksandrov \cite{Alex}, Firey \cite{Firey68}, Caffarelli Nirenberg-Spruck \cite{CNSIV}. For Riemannian manifolds, some results have been obtained by Li-Oliker \cite{LO02} for unit sphere, Barbosa-de Lira-Oliker \cite{BLO} for space forms, Jin-Li \cite{JL06} for hyperbolic space,  Andrade-Barbosa-de Lira \cite{ABdeL2009} for warped product manifolds, Li-Sheng \cite{LiSheng13} for Riemannain manifold equipped with  a global normal Gaussian coordinate system.

For the hypersurface $\Sigma$ in Euclidean space $\Bbb R^{n+1}$, the Weingarten curvature equation in general form is defined by
\begin{eqnarray*}
\sigma_k(\kappa(X))=\psi(X,\nu(X)), \ \ \forall\,  X\in \Sigma,
\end{eqnarray*}
where $\nu(X)$ is the normal vector field along the hypersurface $\Sigma$. In many cases, the curvature estimates are the key part for the above prescribed curvature problems. Let's give a brief review. When $k=1$, curvature estimate comes from the theory of quasilinear PDE.   If $k=n$, curvature estimate in this case for general $\psi(X, \nu)$ is due to Caffarelli-Nirenberg-Spruck \cite{CNS1}.  Ivochkina \cite{I1, I} considered the Dirichlet problem of the above equation on domains in $\mathbb R^n$, and obtained $C^2$ estimates  there under some extra conditions on the dependence of $f$ on $\nu$. $C^2$ estimate was also proved for equation of prescribing curvature measures problem in \cite{GLM, GLL}.
 If the function $\psi$ is  convex respect to the normal $\nu$, it is well
known  that the global $C^2$ estimate has been obtained by B. Guan
\cite{B}. Recently, Guan, Ren and the third author \cite{GRW2015}  obtained global $C^2$ estimates for a closed convex hypersurface $\Sigma\subset \mathbb R^{n+1}$ and then solved the long standing problem (\ref{WeinEqu}). In the same paper \cite{GRW2015}, they also proved the estimate for  starshaped $2$-convex hypersurfaces by introducing some new test curvature functions. In \cite{LRW}, Li, Ren and the third author relax the convex to $(k+1)$- convex for any $k$ Hessian equations. In \cite{RW16}, Ren and the third author totally solved the case  $k=n-1$, that is the global curvature estimates of $n-1$ convex solutions for $n-1$ Hessian equations. In \cite{SX15}, Spruck-Xiao extended $2$-convex case in \cite{GRW2015}  to space forms and give a simple proof for the Euclidean case. We also note the recent important work on the curvature estimates and $C^2$ estimates developed by Guan \cite{B2} and Guan-Spruck-Xiao \cite{GSX}.

 These type of equations and estimates for generalized right hand side appear some new geometric applications recently. In \cite{PPZ1}, \cite{PPZ2}, Phong-Picard-Zhang generalized the Fu-Yau equations, which is a complex $2$-Hessian equations depending on gradient term
 on the right hand side. The \cite{PPZ2}, \cite{PPZ3} obtained their $C^2$ estimates using the idea of \cite{GRW2015}.  In \cite{GL}, Guan-Lu
 consider the curvature estimate for hypersurfaces in high dimensional Riemannian manifolds, which is also a $2$-Hessian equation depending on
 normal.  The estimates in \cite{GRW2015} is also applied in \cite{X} and \cite{BIS}.

 Let $(M^n,g^\prime)$ be a compact Riemannian manifold and $I$ be an open interval in $\Bbb R$. The warped product manifold $\bar{M}=I\times_h M$  is endowed with the metric
\begin{equation}\label{Wmetric}
\bar{g}^2=dt^2+h^2(t)g^\prime,
\end{equation}
where $h: I\longrightarrow \Bbb R^+$ is the positive differential function.
Given a differentiable function $z:M\longrightarrow I$, its graph is defined as  the hypersuface
\begin{equation*}
\Sigma=\{X(u)=(z(u),u)|u\in M \}.
\end{equation*}
For the Weingarten curvature equation in general form
\begin{eqnarray}\label{WeinEqu}
\sigma_k(\kappa(V))=f(\kappa(V))=\psi(V,\nu(V)), \ \ \forall  V\in \Sigma,
\end{eqnarray}
where $V=h\,\dfrac{\p}{\p t}$ is the position vector field of hypersurfce $\Sigma$ in $\bar M$, $\sigma_k$ is the $k^{th}$ elementary symmetric function, $\nu(V)$ is the inward unit normal vector feild along the hypersurface $\Sigma$ and $ \kappa(V)=(\kappa_1,\cdots,\kappa_n)$ are principal curvatures of hypersurface $\Sigma$ at $V$. 
Given $t_-, t_+$ with $t_-<t_+$, we define the annulus domain $\bar{M}_-^+=\{(t,u)\in\bar{M}|t_-\leq t \leq t_+\}$.

In this article, we will generalize the results in  \cite{GRW2015}, \cite{RW16} to the hypersufaces in warped product manifolds.  The main results of this paper are the followings:

\begin{thm}\label{MainThm1}
Let $M^n$ be a compact Riemannian manifold and  $\bar M$ be the warped product manifold with the metric (\ref{Wmetric}). Assume that $h$ is a positive differential function and $h^\prime>0$.	Suppose that $\psi$ satisfies
\begin{enumerate}	
	\item[(a)] $\psi (t,u,\nu(u)) > C^k_n(\kappa(t))^k$ for $t\le t_-$,	
	\item[(b)] $\psi (t,u,\nu(u)) < C^k_n(\kappa(t))^k$ for $t\ge t_+$,
	\item[(c)] $\partial_t \big(h^k\psi(V,\nu)\big) \le 0$ for $t_-<t<t_+$,
\end{enumerate}
where $\kappa(t)=h'(t)/h(t)$ and $C^k_n$ is the combinatial numbers.
Then there exists  a unique  differentiable function $z:M^n\rightarrow I$ solve the equation (\ref{WeinEqu}) for $k=2$ and $k=n-1$ whose graph $\Sigma$ is contained in the interior of the region $\bar{M}_-^+$.
\end{thm}
For the convex hypersurcaes in any warped product manifolds, we obtain the global curvature estimates.
\begin{thm}\label{MainThm2}
Suppose $\Sigma\longrightarrow \bar{M}^{n+1}$ is a  convex compact hypersurface satisfying curvature equation \eqref{WeinEqu} for some positive function $\psi(V, \nu)\in C^{2}(\Gamma)$, where $\Gamma$ is an open neighborhood of unit normal bundle of $M$ in $\bar{M}^{n+1} \times \mathbb S^n$, then there is a constant $C$ depending only on $n, k$, $|z|_{C^1}$, $\inf \psi $ and $\|\psi\|_{C^2}$, such that
	\begin{equation}\label{Mc2-0}
	\max_{u\in M} \kappa_i(u) \le C.
	\end{equation}
\end{thm}

Since the second foundamental form do not satisfy  Codazzi properties for hypersurfaces in warped product in general, the constant rank theorem is still no known. Thus,  the above estimates only can implies  the existence results in sphere.
\begin{thm}\label{MainThm3}
	Let  $\bar M$ be the sphere  with sectional curvature  $\lambda>0$ which means the metric $\bar g$ of $\bar M$ is defined by (\ref{Wmetric}), where function $h$ is defined by 
\begin{equation}\label{Fh}
h(t)=\dfrac{\sin\sqrt{\lambda} t}{\sqrt{\lambda}}.
\end{equation}	
	Suppose that $\psi$ satisfies
	\begin{enumerate}	
		\item[(a)] $\psi (t,u,\nu(u)) > \kappa(t)$ for $t\le t_-$,	
		\item[(b)] $\psi (t,u,\nu(u))< \kappa(t)$ for $t\ge t_+$,
		%\item[(c)] $\partial_t \big(h^{k} \psi(V,\nu)\big) \le 0$ for $t_-<t<t_+$,
		\item[(c)] $(\psi^{-1/k})_{ij}+\lambda\psi^{-1/k}g_{ij}\geq 0,$ for any $\nu$,
		\end{enumerate}
where $\kappa(t)=h'(t)/h(t)=\sqrt{\lambda}\cot(\sqrt{\lambda t})$ and $t_+<\pi/2$. Then there exists  a differentiable function $z:\mathbb{S}^n\rightarrow I$ solve the equation (\ref{WeinEqu}) for  any $k$ whose graph $\Sigma$ is  a strictly convex hypersurface and is contained in the interior of the region $\bar{M}_-^+$.	
\end{thm}

The paper is organized as follows. In Section 2, we fix notation and recall some basic formulae for geometric and analytic preliminaries, including the
detailed description of the problem. In Section 3, we state the $C^0-$estimates without proof under
the hypothesis of the theorem. In Section 4, the gradient estimates of (\ref{WeinEqu})
is presented. In Section 5, the curvature estimates is proved  for the starshaped $2$-convex case. In Section 6 and Section 7,  the $C^2$ estimtes is obtained for  convex and $(n-1)$-convex hypersurface in warped product manifold $\bar M$. The last section derives the constant rank theorem and existence results .

\section{Preliminaries}
\subsection{Warped product manifold $\bar M$}
Let $M^n$ be a compact Riemannian manifold with the metric $g^\prime$ and $I$ be an open interval in $\Bbb R$. Assuming $h: I\longrightarrow \Bbb R^+$ is the positive differential function and $h^\prime>0$,  the manifold $\bar{M}=I\times_h M$ is called the warped product if it is endowed with the metric
\begin{equation}\label{WPmetric}
\bar{g}^2=dt^2+h^2(t)g^\prime.
\end{equation}

In the section, we use Latin lower case letters $i,j,\ldots$ to refer
to indices running from $1$ to $n$ and $a,b,\ldots $ to indices from
$0$ to $n-1$. The  Einstein summation convention is used throughout
the paper.

The metric  in $\bar M$ is denoted by $\langle
\cdot,\cdot \rangle$. The corresponding Riemannian connection in
$\bar M$ will be denoted by $\bar\nabla$. The usual connection in
$M$ will be denoted $\nabla'$. The curvature tensors in $M$ and
$\bar M$ will be denoted by $R$ and $\bar R$, respectively.

Let $e_1, \ldots, e_{n-1}$ be an orthonormal
frame field in  $M$ and  let $\theta_1, \ldots, \theta_n$ be the
associated dual frame. The connection forms $\theta_{ij}$ and
curvature forms $\Theta_{ij}$ in $M$ satisfy the structural
equations
\begin{eqnarray}
& & \dd\theta_i = \sum_j\theta_{ij}\wedge \theta_j,\quad \theta_{ij}=-\theta_{ji}, \\
& &  \dd\theta_{ij} - \sum_k\theta_{ik}\wedge \theta_{kj}=\Theta_{ij}=-\frac12 \sum_{k,l}R_{ijkl}\theta_k\wedge\theta_l.
\end{eqnarray}
An orthonormal frame in $\bar M$ may be defined by $\bar{e}_i =
(1/h)e_i,\, 1\le i\le n-1,$  and $\bar{e}_{0} =
\partial/\partial t$. The associated dual frame is then
$\bar{\theta}_i = h\theta_i$ for $1\le i\le n-1$ and
$\bar{\theta}_{0}=\dd t$.
A simple computation permits to obtain
\begin{lem} On the leaf $M_t$, the curvature satisfies
	\begin{equation}\label{Curvature0}
\bar R_{ijk0}=0
\end{equation}
and the principle curvature is given by
\begin{equation} \label{Pcur1}
\kappa(t) = h'(t)/h(t)
\end{equation}
where the  inward unit normal $-\bar e_0=-\partial/\partial t$ is choosen  for each leaf $M_t$.
\end{lem}

%-----------------------------------------------------------------------
\subsection{Hypersurfaces in warped product manifold $\bar M$}

Given a differentiable function $z:M\longrightarrow I$, its graph is defined as  the hypersuface
\begin{equation}\label{Hyper1}
\Sigma=\{X(u)=(z(u),u)|u\in M \}
\end{equation}
whose tangent space is spanned at each point by the vectors
\begin{equation} \label{tangvect1}
X_i = h\,\bar{e}_i + z_i\,\bar{e}_{0},
\end{equation}
where $z_i$ are the components of the differential $\dd z=
z_i\theta^i$. The unit vector field
\begin{equation}\label{normalvect1}
\nu = \frac{1}{\sqrt{h^2 + |\nabla' z|^2}}\big( \sum_{i=1}^n z^i\bar{e}_i - h\bar{e}_{0}\big)
\end{equation}
is an unit inner normal vector field to $\Sigma$.
Here, $|\nabla' z|^2=z^iz_i$ is the squared norm of $\nabla
'z=z^ie_i$. The components of the induced metric in  $\Sigma$ is given by
\begin{equation}\label{Inmetric}
 g_{ij} = \langle X_i, X_j\rangle = h^2\delta_{ij} +
z_iz_j
\end{equation}
The  second fundamental form  of $\Sigma$ with components $(a_{ij})$   is determined by
\begin{eqnarray*}
	\label{bij1} a_{ij} =\langle \bar\nabla_{X_j}X_i,\nu\rangle=
	\frac{1}{\sqrt{h^2 + |\nabla' z|^2}}\big(-hz_{ij} + 2h'z_iz_j + h^2h'\delta_{ij}\big)
\end{eqnarray*}
where $z_{ij}$ are the components of the Hessian $\nabla'^2z
=\nabla'\dd z$ of $z$ in $M$.

Now we choose the coordinate systems such that  $\{E_0=\nu,E_1,\cdots,E_n\}$ is  an orthonormal
frame field in some open set of  $\Sigma$ and   $\{\omega_0,\omega_1, \ldots, \omega_n\}$ is its associated dual frame. The connection forms $\{\omega_{ij}\}$ and curvature forms $\{\Omega_{ij}\}$ in $\Sigma$ satisfy the structural equations
\begin{eqnarray*}
& & d\omega_i -\sum_j \omega_{ij}\wedge \omega_j=0,\quad \omega_{ij}+\omega_{ji}=0, \\
& &  \dd\omega_{ij} - \sum_k\omega_{ik}\wedge \omega_{kj}=\Omega_{ij}=-\frac12 \sum_{k,l}R_{ijkl}\omega_k\wedge\omega_l.
\end{eqnarray*}
The coefficients $a_{ij}$ of the second fundamental form are given
by Weingarten equation
\begin{equation*}
\label{Weingarten} \omega_{i0}= \sum_ja_{ij}\,\omega_j.
\end{equation*}
The covariant derivative of the second fundamental form $a_{ij}$ in $\Sigma$ is defined by
\begin{eqnarray*}
 \sum_k a_{ijk}\,\omega_k &=& \dd a_{ij}+\sum_la_{il}\,
\omega_{lj}+\sum_la_{lj}\,\omega_{li}, \\
\sum_l a_{ijkl}\,\omega_l&=& \dd a_{ijk}+\sum_l a_{ljk}\,\omega_{li}
+\sum_la_{ilk}\,\omega_{lj}+\sum_la_{ijl}\,\omega_{lk}.
\end{eqnarray*}
The Codazzi equation is a commutation formula for the first order
derivative of $a_{ij}$ given by
\begin{equation}\label{Codazzi}
a_{ijk}-a_{ikj}=-\bar R_{0ijk}
\end{equation}
and the Ricci identity is a  commutation formula for the second order
derivative of $a_{ij}$ given by
\begin{lem} \label{RicciId}
	Let $\bar{X}$ be a point of $\Sigma$ and $\{E_0 =\nu, E_1,\ldots, E_n\}$
	be an adapted frame field such that each $E_i$ is a principal
	direction and $\omega^k_i=0$ at $\bar X$. Let $(a_{ij})$ be the
	second quadratic form of $\Sigma$. Then, at the point $\bar{X}$, we
	have
\begin{equation}\label{RicId4}
\begin{aligned}
a_{llii}=&a_{iill}-a_{lm}\left(a_{mi}a_{il}-a_{ml}a_{ii}\right)-a_{mi}\left(a_{mi}a_{ll}
-a_{ml}a_{li}\right)\\
&+\bar R_{0iil;l}-2a_{ml}\bar{R}_{miil}+a_{il}\bar{R}_{0i0l}+a_{ll}\bar{R}_{0ii0}\\
&+\bar{R}_{0lil;i}-2a_{mi}\bar{R}_{mlil}+a_{ii}\bar{R}_{0l0l}+a_{li}\bar{R}_{0li0}
\end{aligned}
\end{equation}
In particular, we have
\begin{equation}\label{RicId1}
a_{ii11}-a_{11ii}=a_{11}a_{ii}^2-a_{11}^2a_{ii}+2(a_{ii}-a_{11})\bar R_{i1i1}+a_{11}\, \bar
R_{i0i0}-a_{ii} \,\bar R_{1010}+\bar R_{i1i0;1}-\bar
R_{1i10;i}.
\end{equation}
\end{lem}

\subsection{Two functions $\eta$ and $\tau$}
Define the functions $\tau:\Sigma\to \mathbb{R}$ and $\eta:\Sigma\to
\mathbb{R}$  by
\begin{eqnarray}\label{t}
\tau
=-h \ae \nu, \bar e_0\ad=-\ae V, \nu\ad
%=\frac{h^{2}}{W}
\quad \textrm{and}\quad \eta=- \int h\, \dd t,
\end{eqnarray}
where $V=h\bar e_0=h\dfrac{\p}{\p t}$ is the position vector field and $\nu$ is the inner unit  normal. Then we have
\begin{lem}\cite{ABdeL2009} The gradient vector fields of the functions $\eta$ and $\tau$ are
	\begin{equation}\label{eta1}
	\nei\eta=-h\ae \bar{e}_0, E_i \ad E_i
	\end{equation}
	\begin{equation}\label{tau1}
	\nei \tau= - \sum_j\nej\eta a_{ij}
	\end{equation}
	and the second order
	derivative of $\tau$ and $\eta$  are given by
	\begin{equation}\label{eta2}
	\nij \eta=\tau a_{ij}-h^{\prime} g_{ij}
	\end{equation}	
	\begin{equation}\label{tau2}
	\begin{aligned}
	\nij  \tau=& -\sum_k\tau a_{ik}a_{kj}+h^{\prime}a_{ij}-\sum_ka_{ikj}\n_{E_k}\eta\\
	=&-\tau \sum_ka_{ik}a_{kj}+h^{\prime}a_{ij}-\sum_k(a_{ijk}+\bar R_{0ijk})\n_{E_k}\eta.
	\end{aligned}
	\end{equation}
%where $\bar e_0^T$ denotes the tangential projection of the vector field $\bar e_0$.
\end{lem}

\subsection{Basic formulae}
Assume that $\Sigma\longrightarrow \bar{M}$ is the graph defined as
the hypersurface $\Sigma$ whose points are the form
$X(u)=(z(u),u)$ with $u\in M$. This graph is diffeomorphic with $M$
and may be globally oriented by an unit normal vector field $\nu$ for
which it holds that $\ae \nu,\partial_t\ad <0$. Let $\kappa=(\kappa_1, \ldots, \kappa_n)$  be the vector whose components $\kappa_i$ are the principal curvatures of
$\Sigma$, that is, the eigenvalues of the second fundamental form
$B=(\ae \bar{\n}_{i}E_j , \nu\ad)$ in $\Sigma$.

The elementray symmetric function of order $k$ ($1\leq k\leq n$) of  $\kappa=(\kappa_1, \ldots, \kappa_n)$ is as following
\begin{equation}\label{Sigmak1}
\sigma_k=\sum_{i_1<\cdots< i_n}\kappa_{i_1}\cdots\kappa_{i_n}
\end{equation}
Let $\Gamma_k$ be the connected component of $\{\kappa\in  \Bbb R^n|\sigma_m>0, m=1,\cdots, k\}$ containing the positive cone $\{\kappa\in  \Bbb R^n|\kappa_1,\dots,\kappa_n>0\}$.
\begin{defn}\label{AdSolution}
	A positive function $z\in C^2(M^n)$ is said to be admissible  for the operator $\sigma_k$ if for the corresponding hypersurface $\Sigma=\{(z(u),u)|u\in M^n\}$,
	at every point of $\Sigma$ with the normal as in \eqref{normalvect1}, the principal curvatures $\kappa=(\kappa_1,\cdots,\kappa_n)$ is in $\Gamma_k$.
\end{defn}

\begin{lem}(\cite{Andrews,Ball,CNSIII,Gerhardt96})\label{AG}
	Let $F$ be a $C^2$ symmetric function defined in some open set of $Sym(n)$, where $Sym(n)$ is the set of all $n\times n$ symmetric matrices. For any symmetric matrix $(b_{ij})$,    there holds
	$$
	F^{ij, kl}b_{ij}b_{kl}=\sum_{i,j} \frac{\partial^2 f}{\partial
		\kappa_i \partial
		\kappa_j} b_{ii} b_{jj} +\sum_{i\ne j}
	\frac{f_i-f_j}{\kappa_i-\kappa_j}
	b_{ij}^2,
	$$
	where the second term on the right-hand side must be
	interpreted
	as a limit whenever $\kappa_i=\kappa_j$.
\end{lem}

\begin{lem}\cite{GLL,GRW2015}\label{lemma 1}
	Assume that $k>l$, $W=(w_{ij})$ is a Codazzi tensor which is in $\Gamma_k$. Denote $\al=\dfrac{1}{k-l}$.  Then, for $h=1,\cdots, n$, we have the following inequality,
	\begin{eqnarray}\label{1.4}
	&&-\dfrac{\sigma_k^{pp,qq}}{\sigma_k}(W)w_{pph}w_{qqh}+\dfrac{\sigma_l^{pp,qq}}{\sigma_l}(W)
	w_{pph}w_{qqh}\\
	&\geq& \abc{\dfrac{(\sigma_k(W))_h}{\sigma_k(W)}-\dfrac{(\sigma_l(W))_h}{\sigma_l(W)}}
	\abc{(\al-1)\dfrac{(\sigma_k(W))_h}{\sigma_k(W)}-(\al+1)\dfrac{(\sigma_l(W))_h}{\sigma_l(W)}}.\nonumber
	\end{eqnarray}
	Furthermore, for any $\delta>0$, we have
	\begin{eqnarray}\label{1.5}
	&&-\sigma_k^{pp,qq}(W)w_{pph}w_{qqh} +(1-\al+\dfrac{\al}{\delta})\dfrac{(\sigma_k(W))_h^2}{\sigma_k(W)}\\
	&\geq &\sigma_k(W)(\al+1-\delta\al)
	\abz{\dfrac{(\sigma_l(W))_h}{\sigma_l(W)}}^2
	-\dfrac{\sigma_k}{\sigma_l}(W)\sigma_l^{pp,qq}(W)w_{pph}w_{qqh}.\nonumber
	\end{eqnarray}
\end{lem}

%---------------------------------------------------------------------------
\section{Gradient Estimates}
In this section, we follow the idea of \cite{CNSIV} and \cite{GLM}
to derive $C^1$ estimates for the height function $z$. In other words, we are looking for a lower bound for the support function $\tau$. Firstly, we need the following technical assumption:
\be\label{tech}
\frac{\partial}{\partial t}(h(t)^k\psi(V, \nu))\leq 0,\,\,\mbox{where $V=h(t)\dfrac{\p}{\p t}.$}
\ee
\begin{lem}\label{FirstEst}
	Let $\Sigma$ be a graph in $\bar{M}=I\times_h M$ satisfying \eqref{WeinEqu},\eqref{tech} and let $z$ be the height function of $\Sigma$.
	If $h$ has positive lower and upper bounds, then there is a constant C depending on the minimum and maximum values of $z$ such that
	\begin{equation}\label{GEst}
	|\nabla z|\leq C.
	\end{equation}
\end{lem}

\begin{proof}Set  $\chi(z)= -\ln (\tau)+\gamma(-\eta(t))$, where $\gamma$ is a single variable function to be determined later.
	Assume that $\chi$ achieve its maximum value at point $u_0$. We claim that $V$ is parallel to its normal $\nu$ at $u_0$ if we choose suitable $\gamma$. We will prove it by contradiction. If not, we can choose a local orthonormal basis $\{E_i\}_{i=1}^n$ such that $\left<V, E_1\right>\neq 0,$ and $\left<V, E_i\right>=0,\, i\geq 2$. Obviously, $V=\left<V, E_1\right>E_1+\left<V, \nu\right>\nu$. At point $u_0$, by the maximum principal we have
	\begin{eqnarray}\label{G1}
	0&=&\n_{E_i}\chi(z)=-\frac{\nei \tau}{\tau}-\gamma^{\prime}\nei \eta\\
	0&\geq &\nii\chi(z)\\
	&=& -\frac{\nii\tau}{\tau}+\frac{|\nei \tau|^2}{\tau^2}-\gamma^{\prime}\nii\eta+\gamma^{\prime\prime}|\nei \eta|^2\nonumber
	\end{eqnarray}
	From \eqref{tau1}, (\ref{tau2}) and (\ref{G1}), we have
	\begin{equation}\label{G3}
	\begin{aligned}
	0\geq & -\frac{\nii\tau}{\tau}+\frac{|\nei \tau|^2}{\tau^2}-\gamma^{\prime}\nii\eta+\gamma^{\prime\prime}|\nei \eta|^2\\
	=&-\frac{1}{\tau}\left(-\tau a_{il}a_{li}+h^{\prime}a_{ii}-(a_{iil}+\bar R_{0iil})\eta_l\right)+\left(\gamma^{\prime\prime}+(\gamma^{\prime})^2\right)\eta_i^2-\gamma^{\prime}\left(\tau a_{ii}-h^{\prime} g_{ii}\right)
	\end{aligned}
	\end{equation}
	By \eqref{tau1} and (\ref{G1}), we get
	\begin{equation}\label{G4}
	a_{11}=\tau \gamma^{\prime},\qquad a_{i1}=0,\qquad i\geq 2.
	\end{equation}
	Therefore, it is possible to rotate the coordinate system such that $\{E_i\}_{i=1}^n$ are the principal curvature directions of the second fundamental form $(a_{ij})$, i.e. $a_{ij}=a_{ii}\delta_{ij}$, which means that $(\sigma_k^{ij})$ is also diagonal.  By multiplying $\sigma_k^{ii}$ in the inequality (\ref{G3}) both sides and taking sum on $i$ from $1$ to $n$,  one gets from (\ref{G3}) and (\ref{G4})
	\begin{equation}\label{G5}
	\begin{aligned}
	0\geq &\sigma_k^{ii}a_{ii}^2-\frac{1}{\tau}h^{\prime}\sigma_k^{ii}a_{ii}+\frac{1}{\tau}\sigma_k^{ii}(a_{iil}+\bar R_{0iil})\eta_l+\left(\gamma^{\prime\prime}+(\gamma^{\prime})^2\right)\sigma_k^{ii}\eta_i^2-\gamma^{\prime}\left(\tau \sigma^{ii}a_{ii}-h^{\prime} \sum_{i=1}^n\sigma_k^{ii}\right)\\
	=&\sigma_k^{ii}a_{ii}^2+\frac{1}{\tau}\sigma_k^{ii}a_{ii1}\eta_1+\frac{1}{\tau} \sigma_k^{ii}\bar R_{0ii1}\eta_1+\left(\gamma^{\prime\prime}+(\gamma^{\prime})^2\right)\sigma_k^{11}\eta_1^2+\gamma^{\prime}h^{\prime} \sum_{i=1}^n\sigma_k^{ii}-\gamma^{\prime}\tau k\psi-\frac{1}{\tau}h^{\prime}k\psi
	\end{aligned}
	\end{equation}
	where $F^{ii}a_{ii}= k\psi$ is used.
	Differentiating equation \eqref{WeinEqu} with respect to $E_1$ we obtain
	\be\label{G5}
	\sigma_k^{ii}a_{ii1}=d_V\psi(\n_{E_1}V)-a_{11}d_{\nu}\psi(E_1).
	\ee
	
	Putting (\ref{G4}) and (\ref{G5}) into (\ref{G3}) yields
	\begin{equation}\label{G6}
	\begin{aligned}
	0\geq &	\sigma_k^{ii}a_{ii}^2+\frac{1}{\tau}\Big(d_V\psi(\n_{E_1}V)-a_{11}d_{\nu}\psi(E_1)\Big)\eta_1+\frac{1}{\tau} \sigma_k^{ii}\bar R_{0ii1}\eta_1\\
	&+\left(\gamma^{\prime\prime}+(\gamma^{\prime})^2\right)\sigma_k^{11}\eta_1^2+\gamma^{\prime}h^{\prime} \sum_{i=1}^n\sigma_k^{ii}-\gamma^{\prime}\tau k\psi-\frac{1}{\tau}h^{\prime}k\psi\\
	= &	\sigma_k^{ii}a_{ii}^2-\frac{1}{\tau}\Big(kh^{\prime}\psi+\langle V,E_1 \rangle d_V\psi(\n_{E_1}V)\Big)+\gamma^{\prime} d_{\nu}\psi(E_1)\langle V,E_1 \rangle+\frac{1}{\tau} \sigma_k^{ii}\bar R_{0i1i}\langle V,E_1 \rangle\\
	&+\left(\gamma^{\prime\prime}+(\gamma^{\prime})^2\right)\sigma_k^{11}\langle V,E_1 \rangle^2-k\gamma^{\prime}\tau \psi+\gamma^{\prime}h^{\prime} \sum_{i=1}^n\sigma_k^{ii}.
	\end{aligned}
	\end{equation}
	 Since $V=\langle V,E_1\rangle E_1+\langle V, \nu\rangle\nu$, we have
	\begin{equation}\label{G7}
	d_V\psi(V, \nu)=\langle V, E_1\rangle d_V\psi(\nabla_{E_1}V)+\langle V, \nu\rangle d_V\psi(\nabla_{\nu}V).
	\end{equation}
	Putting (\ref{G7}) into (\ref{G6}) gets
	\begin{equation}\label{4.11}
	\begin{aligned}
	0\geq &\sigma_k^{ii}a_{ii}^2-\frac{1}{\tau}\left(kh^{\prime}\psi+ d_V\psi(V, \nu)\right)+\gamma^{\prime} d_{\nu}\psi(E_1)\langle V,E_1 \rangle+\frac{1}{\tau} \sigma_k^{ii}\bar R_{0i1i}\langle V,E_1 \rangle\\
	&+\left(\gamma^{\prime\prime}+(\gamma^{\prime})^2\right)\sigma_k^{11}\langle V,E_1 \rangle^2-k\gamma^{\prime}\tau \psi+\gamma^{\prime}h^{\prime} \sum_{i=1}^n\sigma_k^{ii}+ d_V\psi(\nabla_{\nu}V)\\
	\geq &\sigma_k^{ii}a_{ii}^2+\left(\gamma^{\prime\prime}+(\gamma^{\prime})^2\right)\sigma_k^{11}\langle V,E_1 \rangle^2+\gamma^{\prime}h^{\prime} \sum_{i=1}^n\sigma_k^{ii}+\frac{1}{\tau} \sigma_k^{ii}\bar R_{0i1i}\langle V,E_1 \rangle\\
	&+\gamma^{\prime} d_{\nu}\psi(E_1)\langle V,E_1 \rangle-k\gamma^{\prime}\tau \psi+ d_V\psi(\nabla_{\nu}V),
	\end{aligned}
	\end{equation}
	where we use the assumption \eqref{tech}.
	Choosing the function $\gamma(t)=\dfrac{\alpha}{t}$ for a positive constant $\alpha$, we have
	\begin{equation}\label{}
	\begin{aligned}
	\gamma^{\prime}(t)=-\frac{\alpha}{t^2},\qquad  \gamma^{\prime\prime}(t)=\frac{2\alpha}{t^3}.
	\end{aligned}
	\end{equation}
	By \eqref{G4} and the choice of function $\gamma$, we have $a_{11}\leq 0$.  Thus, the Newton-Maclaurin inequality  imply
	\begin{equation}\label{NM2}
        \sigma_k^{11}\geq \sigma_{k-1}\geq \frac{k}{(n-k+1)(k-1)}(C_n^k)^{\frac1k} \psi^{\frac{k-1}{k}}.
	\end{equation}
	
	Therefore by the previous three inequalities, we have
	\begin{equation}\label{G8}
	\begin{aligned}
	0 \geq &\sigma_k^{11}a_{11}^2+\left(\frac{\alpha^2}{t^4}+\frac{2\alpha}{t^3}\right)\sigma_k^{11}\langle V,E_1 \rangle^2-\frac{\alpha}{t^2}h^{\prime} \sum_{i=1}^n\sigma_k^{ii}+\frac{1}{\tau} \sigma_k^{ii}\bar R_{0i1i}\langle V,E_1 \rangle\\
	&-\frac{\alpha}{t^2} d_{\nu}\psi(E_1)\langle V,E_1 \rangle+\frac{\alpha}{t^2}k\tau \psi+ d_V\psi(\nabla_{\nu}V)
	\end{aligned}
	\end{equation}
	Since $V=\left<V, E_1\right>E_1+\left<V, \nu\right>\nu$,  one can find that $V \perp \text{Span}\{E_2,\cdots,E_n\}$. One the other hand, $E_1, \nu\perp\text{Span}\{E_2,\cdots,E_n\}$. It is possible to choose coordinate systems such that $\bar e_1\perp\text{Span}\{E_2,\cdots,E_n\}$, which implies that the pair $ \{V, \bar e_1\} $ and $\{\nu,E_1\}$ lie in the same plane and $$\text{Span}\{E_2,\cdots,E_n\}=\text{Span}\{\bar e_2,\cdots,\bar e_n\}.$$
	Therefore, we can choose $E_2=\bar e_2,\cdots,E_n=\bar e_n$. The vector $\nu$ and $E_1$ can decompose to
	\begin{equation*}
	\begin{aligned}
	\nu=&\langle \nu,\bar e_0\rangle \bar e_0 +\langle \nu,\bar e_1\rangle \bar e_1
	=-\frac{\tau}{h}\bar e_0 +\langle \nu,\bar e_1\rangle \bar e_1,\\
	E_1=&\langle E_1,\bar e_0\rangle \bar e_0 +\langle E_1,\bar e_1\rangle \bar e_1,\\
	\end{aligned}
	\end{equation*}
	For (\ref{Curvature0}) and $V=\left<V, E_1\right>E_1+\left<V, \nu\right>\nu $, we obtain
	\begin{equation}\label{Curvature1}
	\begin{aligned}
	\bar R_{0i1i}=& \bar R(\nu,E_i,E_1,E_i)\\
	=&-\frac{\tau}{h} \langle E_1,\bar e_0\rangle \bar R(\bar e_0,\bar e_i,\bar e_0,\bar e_i)+\langle \nu,\bar e_1\rangle
	\langle E_1,\bar e_1\rangle \bar R(\bar e_1,\bar e_i,\bar e_1,\bar e_i)\\
	=&-\frac{\tau}{h}\langle E_1,\bar e_0\rangle \bar R(\bar e_0,\bar e_i,\bar e_0,\bar e_i)-\tau \frac{\langle \nu,\bar e_1\rangle^2}{\langle E_1,V\rangle}\bar R(\bar e_1,\bar e_i,\bar e_1,\bar e_i)\\
	=&\tau \left(-\frac{1}{h}\langle E_1,\bar e_0\rangle \bar R(\bar e_0,\bar e_i,\bar e_0,\bar e_i)- \frac{\langle \nu,\bar e_1\rangle^2}{\langle E_1,V\rangle}\bar R(\bar e_1,\bar e_i,\bar e_1,\bar e_i)\right).
	\end{aligned}
	\end{equation}
	The third equality comes from $0=\langle V, \bar e_1 \rangle $. From \eqref{G4}, \eqref{NM2} and (\ref{Curvature1}), \eqref{G8} becomes
	\begin{eqnarray}
	0&\geq& \alpha^2\sigma_k^{11}(\tau^2(\gamma')^2+\frac{\alpha^2}{t^4}\langle V, E_1\rangle^2)-C_1\alpha\sigma_{k-1}-C_2\alpha|d_{\nu}\psi(e_1)|-|d_V\psi(\nabla_{\nu}V)|\nonumber\\
	&\geq& C\alpha^2|V|^2\sigma_k^{11}-C_1\alpha\sigma_k^{11}-C_2\alpha|d_{\nu}\psi(e_1)|-|d_V\psi(\nabla_{\nu}V)|\nonumber
	\end{eqnarray}	
	where $C,C_1,C_2$ depends on  $k,n$, the $C^0$ bound of $h$ and the curvature $\bar R$.
	Thus, we have contradiction when $\alpha$ is large enough. Hence, $V$ parallel to the normal which implies the lower bound of $\tau$.
	\end{proof}

\section{$C^2$ estimates for $\sigma_2$}
In this section,
we  study the solution of the following normalized equation
\be\label{NPE}
F(b)= \left( \begin{array}{c}n\\2 \end{array} \right)^{(-1/2)}\sigma_2(\kappa(a))^{1/2}
=f(\kappa(a_{ij}))=\overline{\psi}(V,\,\nu).
\ee
Now we can prove the $C^2$ estimate for 2-convex hypersurfaces.
\begin{thm}\label{S2C2}
Under the assumption of Theorem \ref{MainThm1},  there is a constant $C$ depending only on $n,k,t_-,t_+$, the $C^1$ bound of $z$  and $|\bar{\psi}|_{C^2}$, such that
	\be\label{p.10}
	\max_{u\in M}|\kappa_i(u)|\leq C.
	\ee
\end{thm}
\begin{proof}
	Define the function
	\begin{equation}\label{Test}
	W(u,\xi)=e^{-\beta \eta}\frac{B(\xi,\xi)}{\tau-a}
	\end{equation}
	where $\tau \geq 2a$ and $\beta$ is a large constant to be chosen, $\xi$ is a tangent vector of $\Sigma$ and $B$ is the second fundamental form.
	Assume that $W$ is achieved  at $X_0=(z(u_0),\,u_0)$ along $\xi$,
	and we may choose a local orthonormal frame $E_1, \dots, E_n$ around $X_0$ such that  $\xi=E_1$ and $a_{ij} (X_0) = \kappa_i \delta_{ij}$,
	where $\kappa_1\geq \kappa_2\geq \ldots\geq \kappa_n$ are the principal curvatures of $\Sigma$
	at $u_0$.	Thus at $u_0,\, \ln W=\ln{a_{11}}-\log{(\tau-a)}-\beta\eta$
	has a local maximum.  Therefore,
	\be\label{p.11}
	0=\frac{a_{11i}}{a_{11}}-\frac{\nabla_i\tau}{\tau-a}-\beta\eta_i,
	\ee
	and
	\be\label{p.12}
	0\geq\frac{a_{11ii}}{a_{11}}-\left(\frac{a_{11i}}{a_{11}}\right)^2
	-\frac{\nabla_{ii}\tau}{\tau-a}+\left(\frac{\nabla_i\tau}{\tau-a}\right)^2-\beta\eta_{ii}.
	\ee
Multiplying  $F^{ii}$  both sides in (\ref{p.12}) and using \eqref{eta1}-\eqref{tau2}, we have
	\be\label{p.13}
	\begin{aligned}
		0\geq& \frac{1}{\kappa_1} F^{ii}a_{11ii}-\frac{1}{\kappa_1^2}F^{ii}\left(a_{11i}\right)^2-
		\frac{1}{\tau-a}F^{ii}\tau_{ii}+F^{ii}\left(\frac{\tau_i}{\tau-a}\right)^2-\beta F^{ii}\eta_{ii}\\
		=&\frac{1}{\kappa_1} F^{ii}a_{11ii}-\frac{1}{\kappa_1^2}F^{ii}\left(a_{11i}\right)^2
		+\frac{\tau}{\tau-a}F^{ii}\kappa_i^2-\frac{h^{\prime}}{\tau-a}\bar\psi+\frac{1}{\tau-a}\sum_lF^{ii}(a_{iil}+\bar R_{0iil})\eta_l\\
		&+\sum_iF^{ii}\left(\frac{\kappa_i\eta_i}{\tau-a}\right)^2-\beta\tau \bar\psi+h^{\prime}\beta\sum_{i=1}^nF^{ii}.
	\end{aligned}
	\ee
The Ricci identity (\ref{RicId4}) yields
	\begin{equation}\label{C1}
	\begin{aligned}
	F^{ii}a_{ii11}-F^{ii}a_{11ii}=&a_{11}F^{ii}a_{ii}^2-a_{11}^2F^{ii}a_{ii}+2F^{ii}(a_{ii}-a_{11})\bar R_{i1i1}+a_{11}\, F^{ii} \bar R_{i0i0}-F^{ii}a_{ii} \,\bar R_{1010}\\
	&+F^{ii}
	\bar R_{i1i0;1}-F^{ii}\bar R_{1i10;i}\\
	\geq& -C_1\kappa_1^2-C_2\kappa_1\sum_i F^{ii},
         \end{aligned}
	\end{equation}
	for sufficient large $\kappa_1$. Inserting (\ref{C1}) into (\ref{p.13}) give
	\begin{equation}\label{C2}
	\begin{aligned}
	0\geq &\frac{1}{\kappa_1} F^{ii}a_{ii11}-\frac{1}{\kappa_1^2}F^{ii}\left(a_{11i}\right)^2
	+\frac{\tau}{\tau-a}F^{ii}\kappa_i^2+\frac{1}{\tau-a}\sum_lF^{ii}(a_{iil}+\bar R_{0iil})\eta_l\\
	&+\sum_iF^{ii}\left(\frac{\kappa_i\eta_i}{\tau-a}\right)^2-C_1\kappa_1+(h'\beta-C_2)\sum_i F^{ii}-C_3(\beta).
        \end{aligned}
	\end{equation}
Taking covariant differential equation (\ref{NPE}) yields
	\begin{equation}\label{CD1}
	F^{ii}a_{iij}=\bar{\psi}_V(\nabla_{E_j}V)-a_{jl}\bar{\psi}_{\nu}(E_l).
	\end{equation}
Taking covariant differential equation (\ref{CD1}) again yields
	\begin{equation}\label{CD2}
	\begin{aligned}
	F^{ii}a_{ii11}+F^{ij,kl}a_{ij1}a_{kl1}=&\bar \psi_{VV}(\n_{E_1}V,\n_{E_1}V)+2a_{1l}\bar \psi_{V\nu}(\n_{E_1}V,E_l)-a_{1l1}\bar \psi_{\nu}(E_l)+a_{1k}a_{1l}\bar\psi_{\nu\nu}(E_l,E_l)\\
	\geq & -C(1+\kappa_1^2)-a_{1l1}\bar \psi_{\nu}(E_l)\\
	=&-C(1+\kappa_1^2)-(a_{11l}-\bar R_{01l1})\bar \psi_{\nu}(E_l)\\
	\geq& -C(1+\kappa_1^2+\beta \kappa_1)-a_{11l}\bar \psi_V(E_l).
	\end{aligned}
	\end{equation}
	where we use the Codazzi equation in the last equality, (\ref{p.11}) and the bound of curvature of  ambient manifold in the last inequality.
	
	We also have
\begin{eqnarray}\label{CD}
\frac{1}{\kappa_1}\sum_l a_{11l}\bar \psi_V(E_l)-\sum_l\frac{\eta_l}{\tau-a}F^{ii}a_{iil}=\sum_l\beta\eta_l\bar \psi_V(E_l)-\sum_l\frac{\eta_l}{\tau-a}\bar{\psi}_V(\nabla_{E_j}V).
\end{eqnarray}
	Combing the inequality (\ref{CD2}) and (\ref{CD}),  (\ref{C2}) gives
	\begin{equation}\label{CD3}
	\begin{aligned}
	0\geq &\frac{1}{\kappa_1} \left(-F^{ij,kl}a_{ij1}a_{kl1} \right)-\frac{1}{\kappa_1^2}F^{ii}\left(a_{11i}\right)^2
	+\frac{\tau}{\tau-a}F^{ii}\kappa_i^2\\
	&+\sum_{i=1}^nF^{ii}\left(\frac{\kappa_i\eta_i}{\tau-a}\right)^2-C_1\kappa_1+(h'\beta-C_2)\sum_i F^{ii}-C_3(\beta)\\
	\end{aligned}
	\end{equation}

	In the following, we consider two cases.
	
	\vspace{0.3cm}
	
	\noindent {\bf  Case  1}  In this case, we suppose that $\kappa_{n}\leq-\theta\kappa_{1}$
	for some positive constant $\theta$ to be chosen later. In this case, using the concavity of $F$, we  discard the term $-\frac{1}{\kappa_1}F^{ij,kl}a_{ij1}a_{kl1}$.
	
	By Young's inequality and (\ref{p.11}), we have
	\begin{equation}\label{Cas1}
	\begin{aligned}
	\frac{1}{\kappa_1^2} F^{ii}|a_{11i}|^2\leq&
	(1+\epsilon^{-1})\beta^2 F^{ii}|\eta_i|^2
	+\frac{(1+\epsilon)}{(\tau-a)^2} F^{ii}|\tau_i|^2\\
	\leq & C_4(1+\epsilon^{-1})\beta^2 \sum_iF^{ii}
	+\frac{(1+\epsilon)}{(\tau-a)^2} F^{ii}|\tau_i|^2
	\end{aligned}
	\end{equation}
	for any $\epsilon>0$, where we use $|\n \eta|\leq C$. From (\ref{CD3}) and (\ref{Cas1}), we obtain
	\begin{equation}\label{Cas2}
	\begin{aligned}
	0\geq & -C_1\kappa_1-C_3(\beta)+\left(\frac{\tau}{\tau-a}-C_5\epsilon\right)
	F^{ii}\kappa_i^2
	+\left(h^{\prime}\beta-C_2-C_4(1+\epsilon^{-1})\beta^2\right)\sum_iF^{ii}\\
	\geq &-\bar C(\kappa_1+\beta)+C_6\sum_{i=1}^nF^{ii}\kappa_i^2-C_7\beta^2 \sum_iF^{ii}.
	\end{aligned}
	\end{equation}
	Since  $F^{11}\leq F^{22}\leq \cdots \leq F^{nn}$ and $\kappa_{n}\leq-\theta\kappa_{1}$, we get
	$$
	\sum_{i=1}^n F^{ii} \kappa_i^2 \geq F^{nn}
	\kappa_n^2\geq \frac{1}{n}\theta^2  \sum_iF^{ii} \kappa_1^2.
	$$
	Hence,
\begin{equation}\label{C2e}
		0\geq -\bar C(\kappa_1+\beta)+\left(C_6\frac{1}{n}\theta^2\kappa_1^2-C_7\beta^2\right) \sum_iF^{ii}.
		\end{equation}
	Since $\sum_iF^{ii}\geq 1$ for sufficient large $\kappa_1$, the inequality (\ref{C2e}) clearly implies the bound of $\kappa_1$ from above.

	\vspace{0.3cm}
	
	\noindent {\bf  Case  2}
	In this case, we assume that
	$\kappa_{n}\geq-\theta\kappa_{1}$. Hence,
	$\kappa_{i}\geq \kappa_n\geq-\theta\kappa_{1}$. We then group the indices in
	$\{1,...,n\}$ in two sets $I=\{j:F^{jj}\leq4F^{11}\}$ and
	$J=\{j:F^{jj}>4F^{11}\}$. By using (\ref{p.11}), we can infer
	\begin{equation}\label{Case21}
	\begin{aligned}
	\frac{1}{\kappa_1^2}\sum_{i\in I} F^{ii}|a_{11i}|^2
	\leq & C_1(1+\epsilon^{-1})\beta^2 F^{11}
	+\frac{(1+\epsilon)}{(\tau-a)^2} F^{ii}|\tau_i|^2
	\end{aligned}
	\end{equation}
	for any $\epsilon>0$. Therefore it follows from (\ref{CD3}) that
	\begin{equation}\label{Case22}
	\begin{aligned}
	0\geq & -C_1\kappa_1-C_3(\beta)-\frac{1}{\kappa_1}F^{ij,kl}a_{ij1}a_{kl1}
	+\left(\frac{\tau}{\tau-a}-C_5\epsilon\right)F^{ii}\kappa_i^2\\
	&+\left(h^{\prime}\beta-C_2\right)\sum_iF^{ii}-\frac{1}{\kappa_1^2}\sum_{i\in J}F^{ii}\left(\nabla_ia_{11}\right)^2- C_4(1+\epsilon^{-1})\beta^2 F^{11}.
	\end{aligned}
	\end{equation}
	Using Lemma (\ref{AG}) and the  Codazzi's
	equation, one gets
	\begin{equation}\label{FCod}
	\begin{aligned}
	-\frac{1}{\kappa_{1}}F^{ij,kl}a_{ij1}a_{kl1} \geq&
	-\frac{2}{\kappa_{1}}\sum_{j\in J}\frac{f_{1}-f_{j}}{\kappa_{1}
		-\kappa_{j}}\big(a_{1j1} \big)^{2}=-\frac{2}{\kappa_{1}}\sum_{j\in J}\frac{f_{1}-f_{j}}{\kappa_{1}
		-\kappa_{j}}\big(a_{11j}-\bar R_{01j1}\big)^{2}.
	\end{aligned}
	\end{equation}
	Following the argument in \cite{JL06}, we may verify that choosing
	$\theta=\dfrac{1}{2}$ it holds that  for all $j\in J$,
	\begin{equation}
	\begin{aligned}\label{fj}
	-\frac{2}{\kappa_{1}}\frac{f_{1}-f_{j}}{\kappa_{1}
		-\kappa_{j}}\geq\frac{f_{j}}{\kappa_{1}^{2}}=\frac{F^{jj}}{\kappa_1^2}.
	\end{aligned}
	\end{equation}
	Combining (\ref{Case22}), (\ref{FCod}) and (\ref{fj}), we obtain
	\begin{equation}\label{5.20}
	\begin{aligned}
	0\geq & -C_1\kappa_1-C_3(\beta)-2\frac{1}{\kappa_{1}^2}\sum_{j\in J}F^{jj} a_{11j}\bar R_{01j1}+\left(\frac{\tau}{\tau-a}-C_5\epsilon\right)F^{ii}\kappa_i^2\\
	&
	+\left(h^{\prime}\beta-C_2\right)\sum_iF^{ii}- C_4(1+\epsilon^{-1})\beta^2 F^{11}\\
	\geq& -\bar C(\kappa_1+\beta)+C_6\sum_{i=1}^nF^{ii}\kappa_i^2
	+\left(h^{\prime}\beta-C\right)\sum_iF^{ii}- C_7\beta^2 F^{11}\\
	\geq&(C_8(h'\beta-C_2)-C_1)\kappa_1+(C_6\kappa_1^2-C_7\beta^2)F^{11}-\bar C_3(\beta),
		\end{aligned}
	\end{equation}
	 by choosing $\epsilon$ small and sufficient large $\kappa_1$. Here we also used \eqref{p.11} and
	$$\sum_{i=1}^n F^{ii}\geq C\kappa_1.$$	For $\beta>0$ sufficiently large,  we may obtain an upper bound for $\kappa_1$ by \eqref{5.20}
\end{proof}

\begin{rem}
The similar idea also has been used in \cite{CW}, \cite{SUW} and \cite{BJ}.
\end{rem}

\section{A global $C^2$ estimate for convex Hypersurface  in warped product space}

In this section, following the arguments in \cite{GRW2015}, we can obtain the  $C^2$ estimates for convex solutions to curvature equation \eqref{WeinEqu} in $\Sigma$, namely, proving Theorem \ref{MainThm2}.
 %-------------------------------------------------------------

	Define the following test function,
	\begin{eqnarray}
 \quad \Psi=\frac12\ln P(\kappa)-N \log \tau-\beta{\eta},
	\end{eqnarray}
	where	$P(\kappa)= \kappa^2_1+\cdots+\kappa_n^2=\sum_{i,j=1}^na_{ij}^2, N,\beta$ is a constant to be determined later.
	
	We assume that $\Psi$ achieves its maximum value at $X_0\in \Sigma$. By a proper
	rotation, we may assume that $(a_{ij})$ is a diagonal matrix at the
	point, and $a_{11}\geq a_{22}\cdots\geq a_{nn}$.
	\par
	At $x_0$, differentiate $\Psi$ twice,
	\begin{eqnarray}\label{3}
	0=\Psi_i= \frac{\sum_{l,j}a_{lj}a_{lji}}{P}-N\frac{\tau_i}{\tau}-\beta \eta_i
	=\frac{\sum_l\kappa_la_{lli}}{P}+N\frac{a_{ii}\eta_i}{\tau}-\beta \eta_i\ \  =   \  \ 0,
	\end{eqnarray}
	and,
	\begin{eqnarray}\label{4}
	\\
	0&\geq & \Psi_{ii}\nonumber\\
	&\geq &\frac{1}{P}\left(\sum_l\kappa_la_{llii}+\sum_{l}a_{lli}^2+\sum_{p\neq q}a_{pqi}^2\right)-\frac{2}{P^2}\left(\sum_l\kappa_la_{lli}\right)^2
	-N\frac{\tau_{ii}}{\tau} +N\frac{\tau_i^2}{\tau^2}-\beta \eta_{ii}\nonumber\\
	&=&\frac{1}{P}\left[\sum_{l}\kappa_l\left(a_{iill}-a_{lm}\left(a_{mi}a_{il}-a_{ml}a_{ii}\right)-a_{mi}\left(a_{mi}a_{ll}
	-a_{ml}a_{li}\right)
	+\bar R_{0iil;l}-2a_{ml}\bar{R}_{miil}+a_{il}\bar{R}_{0i0l}\right.\right.\nonumber\\
	&&+a_{ll}\bar{R}_{0ii0}\left.\left.+\bar{R}_{0lil;i}-2a_{mi}\bar{R}_{mlil}+a_{ii}\bar{R}_{0l0l}+a_{li}\bar{R}_{0li0}\right)+\sum_{l}a_{lli}^2+\sum_{p\neq q}a_{pqi}^2\right]\nonumber\\
	&&-\frac{2}{P^2}\left(\sum_l\kappa_la_{lli}\
	\right)^2+\frac{N}{\tau}\sum_la_{iil}\eta_l
	-\frac{Nh^\prime}{\tau}\kappa_i +N\kappa_i^2+\frac{N}{\tau^2}\kappa_{i}^2\eta_i^2+\frac{N}{\tau}\sum_l\bar{R}_{0iil}\eta_l+\beta (h^{\prime}\delta_{ii}-\tau\kappa_i).\nonumber
	\end{eqnarray}
	Multplying $\sigma_k^{ii}$ both sides gives
	\begin{equation}\label{S1}
	\begin{aligned}
	0\geq & \frac{1}{P}\left[\sum_{l}\kappa_l\left( \sigma_k^{ii}a_{iill}-\sigma_k^{ii}a_{lm}\left(a_{mi}a_{il}-a_{ml}a_{ii}\right)-\sigma_k^{ii}a_{mi}\left(a_{mi}a_{ll}
	-a_{ml}a_{li}\right)\right.\right.\\
	& \left.+\sigma_k^{ii}\bar R_{0iil;l}-2\sigma_k^{ii}a_{ml}\bar{R}_{miil}+\sigma_k^{ii}a_{il}\bar{R}_{0i0l}
	 +\sigma_k^{ii}a_{ll}\bar{R}_{0ii0}+\sigma_k^{ii}\bar{R}_{0lil;i}-2\sigma_k^{ii}a_{mi}\bar{R}_{mlil}+\sigma_k^{ii}a_{ii}\bar{R}_{0l0l}+\sigma_k^{ii}a_{li}\bar{R}_{0li0}\right)\\
	&\left.+\sum_{l}\sigma_k^{ii}a_{lli}^2+\sum_{p\neq q}\sigma_k^{ii}a_{pqi}^2\right]\\
	&-\frac{2}{P^2}\sigma_k^{ii}\left(\sum_l\kappa_la_{lli}\
	\right)^2+\frac{N}{\tau}\sum_l\sigma_k^{ii}a_{iil}\eta_l
	-\frac{Nh^\prime}{\tau}kf +N\sigma_k^{ii}\kappa_i^2+\frac{N}{\tau^2}\sigma_k^{ii}\kappa_{i}^2\eta_i^2\\
	&+\frac{N}{\tau}\sum_l\sigma_k^{ii}\bar{R}_{0iil}\eta_l+\beta \left(h^{\prime}\sum_i\sigma_k^{ii}-\tau kf\right)\\
	\geq &\frac{1}{P}\left[\sum_{l}\kappa_l\sigma_k^{ii}a_{iill}+kf\sum_l\kappa_l^3-C(1+\kappa_1^2)\sum_i\sigma_k^{ii}+\sum_{l}\sigma_k^{ii}a_{lli}^2+\sum_{p\neq q}\sigma_k^{ii}a_{pqi}^2\right]\\
	&-\frac{2}{P^2}\sigma_k^{ii}\left(\sum_l\kappa_la_{lli}\
	\right)^2+\frac{N}{\tau}\sum_l\sigma_k^{ii}a_{iil}\eta_l
	+(N-1)\sigma_k^{ii}\kappa_i^2+\left(C_1\beta-C_2N\right)\sum_i\sigma_k^{ii}-C(\beta,N).
	\end{aligned}
	\end{equation}

	Now differentiate equation \eqref{WeinEqu} twice,

	\begin{eqnarray}\label{6}
	\sigma_k^{ii}a_{iij}&=& d_V\psi(\n_j V) + d_{\nu} \psi( \n_j\nu)\ \ = \ \ h^\prime d_V \psi(E_j)-a_{jl}d_{\nu}\psi(E_l),
	\end{eqnarray}
	and
	\begin{eqnarray}\label{7}
	&&\sigma_k^{ii}a_{iijj}+\sigma_k^{pq,rs}a_{pqj}a_{rsj}\\
	&=&d_V\psi(\n_{jj}V)+d^2_{V}\psi(\n_jV,\n_jV)+2d_Vd_{\nu}\psi(\n_jV,\n_j\nu)
	+d^2_{\nu}\psi(\n_j\nu,\n_j\nu)+d_{\nu}\psi(\n_{jj}\nu).\nonumber\\
	&=&-\frac{h^{\prime\prime}}{h}\eta_j d_V\psi(E_j)+h^\prime a_{jj}d_V\psi(\nu)+(h^\prime)^2 d^2_{V}\psi(E_j,E_j)-2h^\prime a_{jj}d_Vd_{\nu}\psi(E_j,E_j)+a_{jj}^2d^2_{\nu}\psi(E_j,E_j)\nonumber\\
	&&-\sum_la_{ljj}d_{\nu}\psi(E_l)-a_{jj}^2d_{\nu}\psi(\nu)\nonumber\\
	&\geq&-C-C\kappa_j^2-\sum_la_{ljj}d_{\nu}\psi(E_l),\nonumber
	\end{eqnarray}
	where the Schwarz inequality is used in the last inequality.
	
	Since
	\begin{eqnarray}\label{conc}
	-\sigma_k^{pq,rs}a_{pql}a_{rsl}\ \ =\  \ -\sigma_k^{pp,qq}a_{ppl}a_{qql}+\sigma_k^{pp,qq}a_{pql}^2,
	\end{eqnarray}
	it follows from (\ref{3}) and (\ref{6}),
	and Codazzi equation (\ref{Codazzi}) implies
	\begin{equation}\label{Third2}
	\begin{aligned}
	\frac{1}{P}\sum_{l,j}\kappa_ja_{ljj}d_{\nu}\psi(E_l)=&\frac{N}{\tau}\sum_j\sigma_k^{ii}a_{iij}\eta_j -\frac{Nh^\prime}{\tau}\sum_jd_V\psi(E_j)\eta_j+\beta\sum_j\eta_j d_\nu \psi(E_j)\\
	&-\frac1P\sum_{l,j}\kappa_j\bar{R}_{0jlj}d_{\nu}\psi(E_l).
	\end{aligned}
	\end{equation}
	Denote
	\begin{eqnarray}
	&& A_i= \frac{\kappa_i}{P}\left(K(\sigma_k)_i^2-\sum_{p,q}\sigma_k^{pp,qq}a_{ppi}a_{qqi}\right), \ \  B_{i}=2\sum_j\frac{\kappa_j}{P}\sigma_k^{jj,ii}a_{jji}^2, \nonumber \\
	&& C_i=2\sum_{j\neq i}\frac{\sigma_k^{jj}}{P}a_{jji}^2, \ \
	D_i=\frac{1}{P}\sum_j\sigma_k^{ii}a_{jji}^2,\  \
	E_i=\frac{2\sigma_k^{ii}}{P^2}\left(\sum_j \kappa_ja_{jji}\right)^2.\nonumber
	\end{eqnarray}
	By (\ref{6}) and  (\ref{Third2}), we can infer
	\begin{equation}\label{S3}
	\begin{aligned}
	0\geq&\frac{1}{P}\left[\sum_{l}\kappa_l\left(-C-C\kappa_l^2-K(\sigma_k)_l^2+K(\sigma_k)_l^2-\sigma_k^{pp,qq}a_{ppl}a_{qql}+2\sum_{j\not= l}\sigma_k^{ll,jj}a_{ljl}^2\right)\right.\\
	&\left.+kf\sum_l\kappa_l^3+\sum_{l}\sigma_k^{ii}a_{lli}^2+2\sum_{j\neq i}\sigma_k^{ii}a_{iji}^2\right]\\
	&-\frac{2}{P^2}\sigma_k^{ii}\left(\sum_l\kappa_la_{lli}\
	\right)^2
	+(N-1)\sigma_k^{ii}\kappa_i^2+\left(C_1\beta-C_2N-C_3\right)\sum_i\sigma_k^{ii}-C(\beta,N)-\frac{C_4}{\kappa_{1}}\\
	\end{aligned}
	\end{equation}
	From the Codazzi equation $a_{iji}=a_{iij}-\bar R_{0iji}$ and the Cauchy-Schwarz inequality, we have
	\begin{equation*}
	\begin{aligned}
	2\sum_{j\not= l}\sigma_k^{ll,jj}\kappa_la_{ljl}^2=& 2\sum_{j\not= l}\sigma_k^{ll,jj}\kappa_l\left(a_{llj}-\bar R_{0ljl}\right)^2\\
	\geq&  (2-\delta)\sum_{j\not= l}\kappa_l\sigma_k^{ll,jj}a_{llj}^2-C_\delta \sum_{j}\sigma_k^{jj}\\
	=&(2-\delta)\sum_{j\not= l}\kappa_l\sigma_k^{ll,jj}a_{llj}^2-C_\delta \sum_{j\not= l}\kappa_l\sigma_k^{ll,jj},	\end{aligned}
	\end{equation*}
	and
	\begin{equation*}
\begin{aligned}	
	2\sum_{j\neq i}\sigma_k^{ii}a_{iji}^2=2\sum_{j\not= i}\sigma_k^{ii,jj}\left(a_{iij}-\bar R_{0iji}\right)^2
	\geq&(2-\delta)\sum_{j\not= i}\sigma_k^{ii}a_{iij}^2-C_\delta \sum_{i}\sigma_k^{ii},
	\end{aligned}
	\end{equation*}
	where $\delta$ is a small constant to be determainted later and $C_\delta$ is a constant depending on $\delta$.
	Therefore, we obtain
	\begin{equation}\label{S4}
	\begin{aligned}
	0\geq&\frac{1}{P}\left[\sum_{l}\kappa_l\left(-C-C\kappa_l^2-K(\sigma_k)_l^2\right)+kf\sum_l\kappa_l^3\right]\\
	&+(N-1)\sigma_k^{ii}\kappa_i^2+\left(C_1\beta-C_2N-C_3-C_\delta \frac{1}{P}\right)\sum_i\sigma_k^{ii}-C(\beta,N)-\frac{C_4}{\kappa_{1}}\\
	&+\left(1-\frac{\delta}{2}\right)\sum_{i} (A_i+B_i+C_i+D_i-E_i)+\frac{\delta}{2}\sum_{i}\left(A_i+D_i\right)-\frac{\delta}{2}\frac{2}{P^2}\sigma_k^{ii}\left(\sum_l\kappa_la_{lli}\right)^2
	\end{aligned}
	\end{equation}
	In the following, the proof will be divided into two cases.
	
	\noindent Case (A): We have some positive number sequence $\{\delta_i\}^k_{i=1}$. There exists some $2\leq i\leq k-1$, such that  $\kappa_i\geq \delta_i\kappa_1$ and $\kappa_{i+1}\leq \delta_{i+1}\kappa_1$. Choosing $K$ sufficient large, we have $A_i$ is positive by Lemma \ref{lemma 1} . From
	Lemma 4.2, Lemma 4.3 and Corollary 4.4 in \cite{GRW2015}, we can infer
	$$
	\sum_{i} (A_i+B_i+C_i+D_i-E_i)\geq 0.
	$$
	From
	(\ref{3}) and Cauchy-Schwarz inequlity, we have
	\begin{equation}\label{S5}
	\begin{aligned}
	\sigma_k^{ii}\left(\frac{1}{P^2}\sum_l\kappa_la_{lli}\right)^2=
	&\sigma_k^{ii}\left(\frac{N}{\tau}\kappa_i-\beta\right)^2\eta_i^2
	\leq  C_5\left(N^2\sigma_k^{ii}\kappa_i^2+\beta^2\sum_{i}\sigma_k^{ii}\right).
	\end{aligned}
	\end{equation}
	Inserting (\ref{S5}) into (\ref{S4}), we get
	\begin{equation}\label{S6}
	\begin{aligned}
	0\geq&\frac{1}{P}\left[\sum_{l}\kappa_l\left(-C-C\kappa_l^2-K(\sigma_k)_l^2\right)+kf\sum_l\kappa_l^3\right]\\
	&+N\left(\frac 12- C_5\delta N\right)\sigma_k^{ii}\kappa_i^2+ \left(\frac N2-1\right)\sigma_k^{ii}\kappa_i^2\\
	&+\left(C_1\beta-C_2N-C_3-C_\delta \frac{1}{P}-C_5\delta \beta^2\right)\sum_i\sigma_k^{ii}-C(\beta,N)-\frac{C_4}{\kappa_{1}}\\
	\geq & -\frac1P\left(C(K)+C(K)\kappa_1^3\right)+C_6\left(\frac N2-1\right)\kappa_1\\
	&+N\left(\frac 12- C_5\delta N\right)\sigma_k^{ii}\kappa_i^2+ \left(C_1\beta-C_2N-C_3-C_\delta \frac{1}{P}-C_5\delta \beta^2\right)\sum_i\sigma_k^{ii}-C(\beta,N)-\frac{C_4}{\kappa_{1}}\\
	\geq& \left(\frac12 C_6N-C(K)\right)\kappa_1+N\left(\frac 12- C_5\delta N\right)\sigma_k^{ii}\kappa_i^2\\
	&+ \left(C_1\beta-C_2N-C_3-C_\delta \frac{1}{P}-C_5\delta \beta^2\right)\sum_i\sigma_k^{ii}-C(\beta,N)-\frac{C_4}{\kappa_{1}},
	\end{aligned}
	\end{equation}
	where we have used $\sigma_k^{ii}\kappa_i^2\geq c_0\kappa_1$.
	Now let us  choose these constants carefully.  Firstly, choose $N$ such that $$C(K)+1\leq \frac12 C_5N, \text{ and } N\geq 4.$$ Secondly, choose $\beta$ such that $$C_1\beta-C_2N-C_3-3\geq 0 .$$ Thirdly, choose the constant $\delta$ satisfying  $$\max\{N^2,\beta^2\}\leq (2C_5\delta)^{-1}.$$
	At last, take sufficient large $\kappa_1$ satisfying $$\frac{C_\delta}{P}\leq 1.$$ otherwise we are done. Finally, $\kappa_1$ has upper bound by \eqref{S6}.
	
	\noindent Case(B): If the Case(A) does not hold.  That means $\kappa_k\geq \delta_k\kappa_1$. Since $\kappa_l\geq 0$, we have,
	$$\sigma_k\geq \kappa_1\kappa_2\cdots\kappa_k\geq\delta_k^{k-1}\kappa_1^k.$$ This implies the bound of $\kappa_1$.
%\end{proof}

\section{Global curvature estimates for $(n-1)$ convex hypersurface}

For the functions $\tau$ and $\eta$, we introduce the test function which appeared firstly in \cite{GRW2015},
$$
\Psi=\log\log P-N\ln (\tau)-\beta \eta,$$ 
where 
$P=\dsum_le^{\kappa_l}$ and $\{\kappa_{l}\}_{l=1}^n$ are the eigenvalues of the second fundamental form. 

We may assume that the maximum of $\Psi$ is achieved  at some point
$X_0\in \Sigma$. After rotating the coordinates, we may assume the matrix
$(a_{ij})$ is diagonal at the point, and we can further  assume that
$a_{11}\geq a_{22}\cdots\geq a_{nn}$. Denote $\kappa_i=a_{ii}$.

Differentiate the function $\Psi$ twice  at  $X_0$, we have,
\begin{equation}\label{7.1}
\begin{aligned}
0=\Psi_i=&\dfrac{P_i}{P\log P}- N\frac{\tau_i}{\tau }-\beta \eta_i
=\dfrac{1}{P\log P}\sum_le^{\kappa_l}a_{lli}+N\frac{a_{ij}\eta_j}{\tau }-\beta \eta_i,
\end{aligned}
\end{equation}

and
\begin{eqnarray*}
	0&\geq&\Psi_{ii}\\
	&=& \frac{P_{ii}}{P\log P}-\frac{P_i^2}{P^2\log P}-\frac{P_i^2}{(P\log P)^2}-N\frac{\tau_{ii}}{\tau}+N\frac{\tau_i^2}{\tau^2}-\beta
	\eta_{ii}\\
	&=&\frac{1}{P\log P}\left[\sum_le^{\kappa_l}a_{llii}+\sum_le^{\kappa_l}a_{lli}^2+\sum_{\alpha\neq \gamma
	}\frac{e^{\kappa_{\alpha}}-e^{\kappa_{\gamma}}}{\kappa_{\alpha}-\kappa_{\gamma}}a_{\alpha\gamma i}^2-\left(\frac{1}{P}+\frac{1}{P\log P}\right)P_i^2\right]\nonumber\\
	&&+\frac{N}{\tau}\sum_la_{iil}\eta_l-\frac{ Nh^\prime}{\tau}\kappa_i+N\kappa_i^2+\frac{N}{\tau^2}\kappa_i^2\eta_i^2
	+\frac{N}{\tau}\sum_l\bar R_{0iil}\eta_l-\beta (\tau \kappa_i-h^\prime\delta_{ii})\nonumber\\
	&=&\frac{1}{P\log P}\left[\sum_le^{\kappa_l}\left(a_{iill}-a_{lm}\left(a_{mi}a_{il}-a_{ml}a_{ii}\right)-a_{mi}\left(a_{mi}a_{ll}
	-a_{ml}a_{li}\right)
	+\bar R_{0iil;l}\right.\right.\nonumber\\
	&&\left.-2a_{ml}\bar{R}_{miil}+a_{il}\bar{R}_{0i0l}+a_{ll}\bar{R}_{0ii0}+\bar{R}_{0lil;i}-2a_{mi}\bar{R}_{mlil}+a_{ii}\bar{R}_{0l0l}+a_{li}\bar{R}_{0li0}\right)\nonumber\\
	&& \left.+\sum_le^{\kappa_l}a_{lli}^2+\sum_{\alpha\neq \gamma
	}\frac{e^{\kappa_{\alpha}}-e^{\kappa_{\gamma}}}{\kappa_{\alpha}-\kappa_{\gamma}}a_{\alpha\gamma i}^2-\left(\frac{1}{P}+\frac{1}{P\log P}\right)P_i^2\right]\nonumber\\
	&&+\frac{N}{\tau}\sum_la_{iil}\eta_l-\frac{ Nh^\prime}{\tau}\kappa_i+N\kappa_i^2+\frac{N}{\tau^2}\kappa_i^2\eta_i^2
	+\frac{N}{\tau}\sum_l\bar R_{0iil}\eta_l-\beta (\tau \kappa_i-h^\prime\delta_{ii})\nonumber.
\end{eqnarray*}
Contract with $\sigma_{n-1}^{ii}$,
\begin{equation}\label{7.2}
\begin{aligned}
0\geq&\sigma_{n-1}^{ii}\Psi_{ii}\\
=&\frac{1}{P\log
	P}\left[\sum_le^{\kappa_l}\sigma_{n-1}^{ii}a_{iill}+(n-1)\psi\sum_le^{\kappa_l}\kappa_l^2-\sigma_{n-1}^{ii}\kappa_{i}^2\sum_le^{\kappa_l}\kappa_{l}-C(1+\kappa_{1})P\sum_i\sigma_{n-1}^{ii}\right.
\\
&\left.+\sum_l\sigma_{n-1}^{ii}e^{\kappa_l}a_{lli}^2+\sum_{\alpha\neq \gamma
}\frac{e^{\kappa_{\alpha}}-e^{\kappa_{\gamma}}}{\kappa_{\alpha}-\kappa_{\gamma}}\sigma_{n-1}^{ii}a_{\alpha\gamma i}^2-\left(\frac{1}{P}+\frac{1}{P\log P}\right)\sigma_{n-1}^{ii}P_i^2\right]\\
&+\frac{N}{\tau}\sum_l\sigma_{n-1}^{ii}a_{iil}\eta_l-\frac{ 1}{\tau}(n-1)Nh^\prime \psi+N\sigma_{n-1}^{ii}\kappa_i^2\\
&+\frac{N}{\tau^2}\sigma_{n-1}^{ii}\kappa_i^2\eta_i^2+\frac{N}{\tau}\sum_l\sigma_{n-1}^{ii}\bar R_{0iil}\eta_l-(n-1)\beta \tau \psi+h^\prime\beta\sum_i\sigma_{n-1}^{ii}.
\end{aligned}
\end{equation}
Inserting  (\ref{6}), (\ref{7}) into (\ref{7.2}), we obtain
\begin{equation}\label{7.3}
\begin{aligned}
0\geq&\sigma_{n-1}^{ii}\Psi_{ii}\\
\geq &\frac{1}{P\log P}\left[\sum_le^{\kappa_l}\left(-C-C\kappa_1^2-K(\sigma_{n-1})_l^2+K(\sigma_{n-1})_l^2-\sigma_{n-1}^{pq,rs}a_{pql}a_{rsl}\right)\right.\\
&-\sum_{l,j}e^{\kappa_l}a_{jll}d_{\nu}\psi(E_j)+(n-1)\psi\sum_le^{\kappa_l}\kappa_l^2-\sigma_{n-1}^{ii}\kappa_{i}^2\sum_le^{\kappa_l}\kappa_{l}-C(1+\kappa_{1})P\sum_i\sigma_{n-1}^{ii}\\
&\left.+\sum_l\sigma_{n-1}^{ii}e^{\kappa_l}a_{lli}^2+\sum_{\alpha\neq \gamma
}\frac{e^{\kappa_{\alpha}}-e^{\kappa_{\gamma}}}{\kappa_{\alpha}-\kappa_{\gamma}}\sigma_{n-1}^{ii}a_{\alpha\gamma i}^2-\left(\frac{1}{P}+\frac{1}{P\log P}\right)\sigma_{n-1}^{ii}P_i^2\right]\\
&+\frac{N}{\tau}\sum_l\sigma_{n-1}^{ii}a_{iil}\eta_l+N\sigma_{n-1}^{ii}\kappa_i^2+\frac{N}{\tau^2}\sigma_{n-1}^{ii}\kappa_i^2\eta_i^2
+(h^\prime\beta-C)\sum_i\sigma_{n-1}^{ii}-C(\beta, N).
\end{aligned}
\end{equation}
By (\ref{7.1}) and (\ref{6}), and the Codazzi equation (\ref{Codazzi}), we have
\begin{equation}\label{e2.19}
\begin{aligned}
\frac{1}{P\log P}\sum_{l,j}e^{\kappa_l}a_{jll} d_{\nu}\psi(E_j)
=&\frac{N}{\tau}\sum_l\sigma_{n-1}^{ii}a_{iil}\eta_l
-\frac{Nh^\prime}{\tau}\sum_ld_V\psi(E_l)\eta_l\\
&+\beta\sum_j\eta_jd_{\nu}\psi(E_j)-\frac{1}{P\log P}\sum_{l,j}e^{\kappa_l}\bar{R}_{0ljl} d_{\nu}\psi(E_j).
\end{aligned}
\end{equation}
By using \eqref{conc} and \eqref{7.3}, we get
\begin{equation}\label{e2.20}
\begin{aligned}
0\geq &\frac{1}{P\log P}\left[\sum_le^{\kappa_l}\left(-C-C\kappa_1^2-K(\sigma_{n-1})_l^2\right)+\sum_le^{\kappa_i}\left(K(\sigma_{n-1})_i^2-\sigma_{n-1}^{pp,qq}a_{ppi}a_{qqi}\right)\right.\\
&+2\dsum_{l\neq
	i}\sigma_{n-1}^{ii,ll}e^{\kappa_l}a_{lil}^2
+(n-1)\psi\sum_le^{\kappa_l}\kappa_l^2-\sigma_{n-1}^{ii}\kappa_{i}^2\sum_le^{\kappa_l}\kappa_{l}-C(1+\kappa_{1})P\sum_i\sigma_{n-1}^{ii}\\
&\left.+\sum_l\sigma_{n-1}^{ii}e^{\kappa_l}a_{lli}^2+2\sum_{l\neq i
}\sigma_{n-1}^{ii}\frac{e^{\kappa_{l}}-e^{\kappa_{i}}}{\kappa_{l}-\kappa_{i}}a_{li l}^2-\left(\frac{1}{P}+\frac{1}{P\log P}\right)\sigma_{n-1}^{ii}P_i^2\right]\\
%&&+\frac{Nh^\prime}{\tau}\sum_ld_V\psi(E_l)\eta_l+\frac{1}{P\log P}\sum_le^{\kappa_l}\bar{R}_{0ljl} d_{\nu}\psi(E_j)-\frac{ 1}{\tau}(n-1)Nh^\prime \psi\nonumber\\
&+N\sigma_{n-1}^{ii}\kappa_i^2+\frac{N}{\tau^2}\sigma_{n-1}^{ii}\kappa_i^2\eta_i^2
+(\beta h^\prime-C)\sum_i\sigma_{n-1}^{ii}-C(\beta,N)-\frac{C}{\kappa_1}.
\end{aligned}
\end{equation}
From the Codazzi equation (\ref{Codazzi}) and Cauchy-Schwarz inequality, we have
\begin{equation*}
\begin{aligned}
2(a_{lil})^2=&2(a_{lli}-\bar{R}_{0lil})^2\geq(2-\delta)a_{lli}^2-C_\delta,
\end{aligned}
\end{equation*}
where $\delta$ is a small constant to be determined later.
Denoting
\begin{eqnarray}
&&A_i=e^{\kappa_i}\left(K(\sigma_{n-1})_i^2-\sum_{p\neq q}\sigma_{n-1}^{pp,qq}a_{ppi}a_{qqi}\right), \ \  B_i=2\sum_{l\neq i}\sigma_{n-1}^{ii,ll}e^{\kappa_l}a_{lli}^2, \nonumber \\  &&C_i=\sigma_{n-1}^{ii}\sum_le^{\kappa_l}a_{lli}^2; \  \ D_i=2\sum_{l\neq
	i}\sigma_{n-1}^{ll}\frac{e^{\kappa_l}-e^{\kappa_i}}{\kappa_l-\kappa_i}a_{lli}^2,
\ \ E_i=\frac{1+\log P}{P\log P}\sigma_{n-1}^{ii}P_i^2\nonumber,
\end{eqnarray}
we have, by \eqref{e2.20},
\begin{equation}\label{e2.21}
\begin{aligned}
0\geq  &\left(1-\frac12\delta\right)\frac{1}{P\log P}\Big[A_i+B_i+C_i+D_i-E_i\Big]+\frac \delta 2\frac{1}{P\log P}\sum_i(A_i+C_i)\\
&-\frac \delta 2 \frac{1+\log P}{(P\log P)^2}\sigma_{n-1}^{ii}P_i^2
%&&+\frac{1}{P\log P}\left[\sum_le^{\kappa_l}\left(-C-C\kappa_1^2-K(\sigma_{n-1})_l^2+(n-1)\psi\kappa_{l}^2\right)\right]\\
-\frac{C_\delta}{P\ln P}\dsum_{l\neq
	i}\sigma_{n-1}^{ii,ll}e^{\kappa_l}-\frac{C_\delta}{P\ln P}\sum_{l\neq i
}\sigma_{n-1}^{ii}\frac{e^{\kappa_{l}}-e^{\kappa_{i}}}{\kappa_{l}-\kappa_{i}}\\
&+\left(N-1\right)\sigma_{n-1}^{ii}\kappa_i^2+\frac{N}{\tau^2}\sigma_{n-1}^{ii}\kappa_i^2\eta_i^2
+\left(\beta h^\prime- C  \right)\sum_i\sigma_{n-1}^{ii}-C\kappa_1-C(\beta,N,K)-\frac{C}{\kappa_1}.
%&&+\frac{Nh^\prime}{\tau}\sum_ld_V\psi(E_l)\eta_l+\frac{1}{P\log P}\sum_le^{\kappa_l}\bar{R}_{0ljl} d_{\nu}\psi(E_j)-\frac{ 1}{\tau}(n-1)Nh^\prime \psi-(n-1)\beta \tau \psi\nonumber
\end{aligned}
\end{equation}
By Schwarz inequality,  we always have
$$
P_i^2=\left(\sum_le^{\kappa_l}a_{lli}\right)^2\leq P\sum_le^{\kappa_l}a^2_{lli} ,
$$
 which implies
\begin{eqnarray}\label{7.6}
\frac{\delta}{2}\frac{1}{P\log P}C_i\geq \frac{\delta}{2}\frac{\log P }{(P\log P)^2}\sigma_{n-1}^{ii}P_i^2.
\end{eqnarray}
We also have
\begin{eqnarray}\label{7.7}
 \dsum_{l\neq i}\sigma_{n-1}^{ii,ll}e^{\kappa_l} \leq P\sum_{l\neq i}\sigma_{n-1}^{ii,ll}=2P\sum_i\sigma_{n-2}^{ii}=6P\sigma_{n-3}.
 \end{eqnarray}
We divided several cases to compare with $\sigma_{n-2}$. \\
\noindent Case (A)  If $\sigma_{n-2}\geq \sigma_{n-3}$, by \eqref{7.7}, we have, for $n\geq 3$,
\begin{eqnarray}\label{inequality}
\frac{C_{\delta}}{P\log P}\dsum_{l\neq i}\sigma_{n-1}^{ii,ll}e^{\kappa_l} \leq 3n^2 \left(\frac{C_{\delta}}{\log P}\sum_i\sigma_{n-1}^{ii}+1\right).
 \end{eqnarray}
 \noindent Case (B) If $\sigma_{n-2}\leq \sigma_{n-3}$, in $\Gamma_{n-1}$ cone, since $|\kappa_n|\leq \kappa_1/(n-1)$ by the argument in \cite{RW16}, we have
 $$\kappa_1\cdots\kappa_{n-2}\leq C_0\kappa_1\cdots\kappa_{n-3},$$ which implies $\kappa_{n-2}\leq C_0$. We further divide two sub-cases to discuss for index $l=1,\cdots,n$.\\
 \noindent Subcase (B1) If $2|\kappa_l|\leq \kappa_1$,  we have
 $$
 \frac{e^{\kappa_l}}{P}\leq e^{\kappa_l-\kappa_1}\leq e^{-\frac{\kappa_1}{2}}\leq \left[\dfrac{1}{(n-3)!}\left(\dfrac{\kappa_1}{2}\right)^{n-3}\right]^{-1}.
 $$
  The last inequality comes from
 Taylor expansion. Thus, we have
$$\frac{C_{\delta}}{P\log P}\sigma_{n-1}^{ii,ll}e^{\kappa_l}\leq C_1\frac{C_{\delta}}{\kappa_1}\leq 1,$$ for sufficient large $\kappa_1$. \\
\noindent  Subcase (B2) For sufficient large $\kappa_1$, if $2|\kappa_l|\geq \kappa_1$, by $\kappa_{n-2}\leq C_0$, we have $1\leq l\leq n-3$. In this case, we have
$$\sigma_{n-1}^{ii,ll}\leq C_1\kappa_1\cdots\kappa_{l-1}\kappa_{l+1}\cdots\kappa_{n-2}\leq \kappa_1\cdots\kappa_{l-1}\kappa_l\kappa_{l+1}\cdots\kappa_{n-2}\leq \sigma_{n-2}.$$
The middle inequality comes from $2\kappa_l\geq \kappa_1\geq 2C_1$ for sufficient large $\kappa_1$. Thus, we have
$$\frac{C_{\delta}}{P\log P}\sigma_{n-1}^{ii,ll}e^{\kappa_l}\leq \frac{C_{\delta}}{\log P}\sigma_{n-2}.$$
Combing (B1) and (B2) cases, we also have \eqref{inequality}.

 By mean value theorem, we have  some $ \xi$ between in $\kappa_i$ and $\kappa_l$ satisfying
\begin{eqnarray}\label{7.9}
\sum_{l\neq i}\sigma_{n-1}^{ii}\frac{e^{\kappa_{l}}-e^{\kappa_{i}}}{\kappa_{l}-\kappa_{i}}=\sum_{l\neq i}\sigma_{n-1}^{ii}e^{\xi}\leq (n-1)P\sum_i\sigma_{n-1}^{ii}
\end{eqnarray}
Hence, using the discussion in \cite{RW16}, we have $$A_i+B_i+C_i+D_i-E_i\geq 0.$$ Thus, by \eqref{e2.21}, \eqref{7.6}, \eqref{inequality}, \eqref{7.9}, we have
\begin{eqnarray}\label{e2.22}
0&\geq&-\frac \delta 2 \frac{1}{(P\log P)^2}\sigma_{n-1}^{ii}P_i^2+\frac{\delta}{2}\frac{1}{P\log P}\sum_iA_i
%&&+\frac{1}{P\log P}\left[\sum_le^{\kappa_l}\left(-C-C\kappa_1^2-K(\sigma_{n-1})_l^2+(n-1)\psi\kappa_{l}^2\right)\right]\nonumber\\
+\left(N-1\right)\sigma_{n-1}^{ii}\kappa_i^2+\frac{N}{\tau^2}\sigma_{n-1}^{ii}\kappa_i^2\eta_i^2\\
&&+\left(\beta h^\prime- C-\frac{C_{\delta}}{\log P}  \right)\sum_i\sigma_{n-1}^{ii}-C(\beta,N,K)-\frac{C}{\kappa_1}.\nonumber
%&&+\frac{Nh^\prime}{\tau}\sum_ld_V\psi(E_l)\eta_l+\frac{1}{P\log P}\sum_le^{\kappa_l}\bar{R}_{0ljl} d_{\nu}\psi(E_j)-\frac{ 1}{\tau}(n-1)Nh^\prime \psi-(n-1)\beta \tau \psi.\nonumber
\end{eqnarray}
From \eqref{7.1} and the Cauchy-Schwarz inequality, we have
\begin{equation}\label{7.11}
\begin{aligned}
\frac \delta 2 \sigma_{n-1}^{ii}\frac{P_i^2}{(P\log P)^2}=&\frac \delta 2 \sigma_{n-1}^{ii}\left(\frac{N}{\tau}\kappa_i-\beta\right)^2\eta_i^2
\leq C\delta  \left(N^2\sigma_{n-1}^{ii}\kappa_i^2+\beta^2\sum_i\sigma_{n-1}^{ii}\right).
\end{aligned}
\end{equation}
Therefore, by Lemma \ref{lemma 1}, \eqref{e2.22}, \eqref{7.11}, we obtain
\begin{equation}\label{7.12}
\begin{aligned}
0\geq&  % \frac{1}{P\log P}\left[\sum_le^{\kappa_l}\left(-C-C\kappa_1^2-K(\sigma_{n-1})_l^2+(n-1)\psi\kappa_{l}^2\right)\right]\\
\left(N-1- C\delta N^2\right)\sigma_{n-1}^{ii}\kappa_i^2+\left(\beta h^\prime- C-\frac{C_{\delta}}{\log P}- C\delta \beta^2\right)\sum_i\sigma_{n-1}^{ii}-C(\beta,N,K)-\frac{C}{\kappa_1}.
%&+\frac{Nh^\prime}{\tau}\sum_ld_V\psi(E_l)\eta_l+\frac{1}{P\log P}\sum_le^{\kappa_l}\bar{R}_{0ljl} d_{\nu}\psi(E_j)-\frac{ 1}{\tau}(n-1)Nh^\prime \psi-(n-1)\beta \tau \psi.
\end{aligned}
\end{equation}
Since $\sigma_{n-1}^{ii}\kappa_i^2\geq C_1\kappa_1$, we only need to choose the constants $N,\beta,\delta$ carefully.   At first, we take constant $N$ satisfying
$$(N-n-1)C_1-C(K)\geq 1.$$ Secondly, we choose constant $\beta$ satisfying
$$\beta h^\prime-2C-2\geq C_2\beta-2C-2\geq 0.$$ Thirdly, we let constant $\delta$ satisfying
$$\max\{CC_1\delta N^2,C\delta\beta^2\}\leq 1.$$ At last, we take sufficient large $\kappa_1$ satisfying
$$\frac{C_{\delta}}{\log P}\leq \frac{C_{\delta}}{\kappa_1}\leq 1.$$
Finally, by \eqref{7.12}, we obtain the upper bound of $\kappa_{1}$.

\section{The existence results}

We are in the position to give the proof of the existence Theorems.
\begin{proof}[Proof of Theorem \ref{MainThm1}]
We use continuity method to solve the existence result.  For parameter $0\leq s\leq 1$, according to \cite{CNSIV}, \cite{ABdeL2009}, we consider the family of functions,
\begin{eqnarray}
\psi^s(V,\nu)=s\psi(t,u)+(1-s)\varphi(t)\sigma_k(\kappa(t)).\nonumber
\end{eqnarray}
where $\kappa(t)=h^\prime/h$ and $\varphi$ is a positive function defined on $I$ satisfying (i) $\varphi >0$;
	 (ii) $\varphi(t)>1$  for  $t\leq t_{-}$;
 (iii) $\varphi(t) < 1$ for $t\ge t_+$;
and 	(iv) $\varphi^\prime(t)<0$.
It is obvious that there exists a unique point $t_0\in (t_-,t_+)$ such that $\varphi(t_0)=1$. By \cite{ABdeL2009}, $z=t_0$ is the unique hypersurface satisfying problem \eqref{WeinEqu}
and one can check directly that $\psi^s$ also satisfies (a), (b), (c) in Theorem \ref{MainThm1}.  The hight estimate can easily obtained by comparison principal.

The openness and uniqueness also similar as \cite{CNSIV, GRW2015}. In view of Evans-Krylov theory, we only need hight, gradient and $C^2$ estimates to complete the closedness part which has been done in section 3, section 5 and section 6. We complete our proof.
\end{proof}

In what following, we discuss the constant rank theorem in space form according \cite{Guan}, \cite{GMa}, \cite{GLMa}. We rewrite our equation \eqref{WeinEqu} to be
\begin{eqnarray}\label{FF}
F(a)=-\sigma_k^{-1/k}(a)=-\psi^{-1/k}(X,\nu).
\end{eqnarray}

\begin{prop}\label{Prop}Suppose the ambient space  $(\bar M,\bar g)=(\mathbb S^{n+1}, \bar g)$ is the sphere with metric defined by \eqref{Wmetric} and $h(t)$ is given by \eqref{Fh}.
 Suppose  some compact hypersurface $\Sigma$ satisfies \eqref{FF} and its second fundamental form is non-negative definite. Let $X,Y$ be two vector fields in the ambient space and  $\bar\nabla$ be the covaraint diravative of the ambient space.  If the function $\psi$ locally satisfies
$$
\bar\nabla_{X}\bar\nabla_Y(\psi^{-1/k})+\lambda\psi^{-1/k}\bar g_{X,Y}\geq 0,
$$
 at any $(u,z)\in\Sigma$, then the hypersurfce $\Sigma$ is of constant rank. \end{prop}
\begin{proof} According to \cite{Guan},
 suppose $P_0$ is the point where the second fundamental form is of the minimal rank $l$. Let $O$
be some open neighborhood of $P_0$.  If $O$ is sufficient small, we can pick constant A  as in \cite{Guan}. Then we use the test function $\varphi=\sigma_{l+1}(a)+A\sigma_{l+2}(a)$
to establish a differential inequality.

Now we choose local orthonormal frame $\{e_1\cdots,e_n\}$ in the hypersurface $\Sigma$. Since $\bar M$ is the sphere with sectional curvatrue $\lambda$, we obviously have
$$\bar{R}_{abcd}=\lambda (\delta_{ac}\delta_{bd}-\delta_{ad}\delta_{bc}).$$ By Lemma \ref{RicciId}, we have
\begin{eqnarray}
\varphi_j&=&(\sigma_{l+1}^{ii}+A\sigma_{l+2}^{ii})a_{iij},\\
 \varphi_{jj}&=&(\sigma_{l+1}^{ii}+A\sigma_{l+2}^{ii})a_{iijj}+(\sigma_{l+1}^{pq,rs}+A\sigma_{l+2}^{pq,rs})a_{pqj}a_{rsj}\nonumber\\
 &=&(\sigma_{l+1}^{ii}+A\sigma_{l+2}^{ii})[a_{jjii}-a_{im}(a_{mj}a_{ji}-a_{mi}a_{jj})-a_{mj}(a_{mj}a_{ii}-a_{mi}a_{ij})\nonumber\\
&&-2a_{mi}\lambda(\delta_{mj}\delta_{ij}-\delta_{mi}\delta_{jj})+a_{ji}\lambda\delta_{00}\delta_{ji}+a_{ii}\lambda(-\delta_{00}\delta_{jj})\nonumber\\
&&-2a_{mj}\lambda(\delta_{mj}\delta_{ii}-\delta_{mi}\delta_{ij})+a_{jj}\lambda \delta_{00}\delta_{ii}-a_{ij}\lambda\delta_{00}\delta_{ij} ]+(\sigma_{l+1}^{pq,rs}+A\sigma_{l+2}^{pq,rs})a_{pqj}a_{rsj}\nonumber.
\end{eqnarray}
 Thus, we have
\begin{eqnarray}
F^{jj}\varphi_{jj}&=&F^{jj}(\sigma_{l+1}^{ii}+A\sigma_{l+2}^{ii})[a_{jjii}+a^2_{ii}a_{jj}-a^2_{jj}a_{ii}+\lambda a_{ii}\delta_{jj}-\lambda a_{jj}\delta_{ii} ]\\
&&+F^{jj}(\sigma_{l+1}^{pq,rs}+A\sigma_{l+2}^{pq,rs})a_{pqj}a_{rsj}\nonumber\\
&=&(\sigma_{l+1}^{ii}+A\sigma_{l+2}^{ii})[(-\psi^{-1/k})_{ii}-\lambda \psi^{-1/k}\delta_{ii}]-(\sigma_{l+1}^{ii}+A\sigma_{l+2}^{ii})F^{pq,rs}a_{pqi}a_{rsi}\nonumber\\
&&+F^{jj}(\sigma_{l+1}^{pq,rs}+A\sigma_{l+2}^{pq,rs})a_{pqj}a_{rsj}\nonumber.
\end{eqnarray}
Since  the second fundamental form still satisfies Coddazi property in space form case, the process of dealing with the third order terms is same as \cite{Guan}, We also have
$$(\psi^{-1/k})_{,ii}=(\psi^{-1/k})_{ii}-a_{ii}(\psi^{-1/k})_{\nu},$$ where the comma in the first term means taking covariant derivative in the ambient space. Thus, since the index $i$ is bad indices, the third term is useless. We have our results.
\end{proof}

Now, we can prove Theorem \ref{MainThm3}.
\begin{proof}[ Proof of Theorem \ref{MainThm3}]
The proof  also use the degree theory by modifying the proof in \cite{GRW2015}. We consider the auxiliary equation
\begin{eqnarray}\label{4.10}
\sigma_k(\kappa(X))=\psi^s=\big(s\psi^{-1/k}(X,\nu)+(1-s)\bar \varphi ^{-1/k}\big)^{-k},
\end{eqnarray}
where $\bar \varphi$ is defined by  $\bar \varphi =C^k_n\varphi \kappa^k(t)$.
We claim that, for sphere case,
\begin{eqnarray}\label{claim}
(\bar \varphi^{-1/k})_{,ij}+\lambda \bar\varphi^{-1/k}\bar g_{ij}\geq 0.
\end{eqnarray}
where   $\{\bar e_0, \cdots, \bar e_n\}$ is the local orthonomal frame on $\bar M$ .
If the claim holds, by our condition, it is obvious that the $\psi^s$ satisfying condition (c) for parameter $0\leq s\leq 1$. By Proposition \eqref{Prop},  the strictly convexity is preserved along the flow $\psi^s$.

Now, let's discuss Claim \eqref{claim}. Define $\alpha(t)=(C^k_n\varphi)^{1/k}$. Since $\bar\varphi$ is some constant on every slice, we have, for $i,j=1,\cdots, n$
$$(\bar \varphi^{-1/k})_{,ij}=(\bar\varphi^{-1/k})_{ij}-a_{ij}(\bar\varphi^{-1/k})_t=-h^2(t)\kappa(t)\frac{\alpha'(t)\kappa(t)+\alpha(t)\kappa'(t)}{\alpha^2(t) \kappa^2(t)}g'_{ij}=-h^2(t)\frac{\alpha'(t)\kappa(t)+\alpha(t)\kappa'(t)}{\alpha^2(t) \kappa(t)}g'_{ij}.$$
Thus, in space form, we have $$\frac{h''(t)}{h(t)}=\kappa^2(t)-\frac{1}{h^2(t)}, \kappa'(t)=\frac{h''(t)}{h(t)}-\kappa^2(t)=-\frac{1}{h^2(t)}.$$  Then, we have
$$(\bar \varphi^{-1/k})_{,ij}=-h^2(t)\frac{\alpha'(t)\kappa(t)+\alpha(t)\kappa'(t)}{\alpha^2(t) \kappa(t)}g'_{ij}=\frac{\alpha(t)-\alpha'(t)h^2(t)\kappa(t)}{\alpha^2(t)\kappa(t)}g'_{ij}>0,$$
since $\alpha'<0$.

For unit (namely, $\lambda=1$) sphere case, it is easy to see that $$(\bar\varphi^{-1/k})_{tt}=\frac{\p^2 (\bar\varphi^{-1/k})}{\p t^2}=-\frac{\p}{\p t} \frac{\alpha'(t)\kappa(t)+\alpha(t)\kappa'(t)}{\alpha^2(t) \kappa^2(t)}= \frac{\p}{\p t}\frac{\alpha-\alpha' \sin t\cos t}{\alpha^2\cos^2 t}.$$ Thus, we have
\begin{eqnarray}
(\bar \varphi^{-1/k})_{tt}&=& \frac{2\alpha'\sin^2 t-\alpha''\sin t\cos t}{\alpha^2\cos^2 t}-\frac{\alpha-\alpha'\sin t \cos t}{\alpha^4\cos^4 t}\left[(-2\cos t\sin t)\alpha^2+2\alpha\alpha'\cos^2 t\right]\nonumber\\
&\geq&\frac{2\alpha'\sin^2 t}{\alpha^2 \cos^2 t}+2\frac{\alpha-\alpha'\sin t \cos t}{\alpha^2 \cos^4 t}\cos t\sin t\nonumber\\
&=&\frac{2\alpha\cos t\sin t}{\alpha^2 \cos^4 t}\nonumber\\
&>&0\nonumber,
\end{eqnarray}
if we require $\alpha'<0$ and $\alpha'' >0$. Thus, Claim \eqref{claim} holds for unit sphere. Since it is rescaling invariant, then \eqref{claim} holds for any $\lambda>0$.

Now we can give the requirements for functions $\varphi(t)$ to satisfying $\alpha'<0$ and $\alpha''>0$. It is a direct calculation that
$$k\alpha^{k-1}\alpha'=C^k_n\varphi'; \text{ and } k\alpha^{k-1}\alpha''+k(k-1)\alpha^{k-2}(\alpha')^2=C^k_n\varphi'',$$ which implies that
\begin{eqnarray}\label{7.6}
\varphi'<0; \text{ and }  \varphi\varphi''>\frac{k-1}{k}(\varphi')^2.
\end{eqnarray}
 We further need that (i) $\varphi >0$;
	 (ii) $\varphi(t)>1$  for  $t\leq t_{-}$;
 (iii) $\varphi(t) < 1$ for $t\ge t_+$. There is a lot of functions satisfying (i)(ii)(iii) and \eqref{7.6}, for example $$\varphi(t)=\exp(\frac{t_-+t_+}{2}-t).$$ Thus, the initial surfaces satisfy the condition of constant rank theorem and height estimates comes from comparing principal.  The  curvature estimates has been obtained in Section 5. The rest part of this proof is similar to convex case in Euclidean space, where we only need to replace the constant rank theorem in \cite{GRW2015} by Proposition 8.1 here.
\end{proof}

\begin{rem}
In hyperbolic space, the problem is that the slice spheres can not satisfying the constant rank theorem: Proposition \ref{Prop}. It may be  an interesting problem to find some other nontrivail initial family of hypersurfaces to satisfy   Proposition \ref{Prop}.
\end{rem}

\noindent {\it Acknowledgement:} The last author wish to thank  Professor Pengfei Guan for some discussion about constant rank theorems. He also would like to thank Tsinghua University for their support and hospitality during the paper being  prepared.

\providecommand{\bysame}{\leavevmode\hbox
	to3em{\hrulefill}\thinspace}

\begin{flushleft}
	Daguang Chen\\
	Department of Mathematical Sciences, Tsinghua University, Beijing, 100084, P.R. China \\
	E-mail: dgchen@math.tsinghua.edu.cn\\
\end{flushleft}

\begin{flushleft}
Haizhong Li,\\
Department of Mathematical Sciences, Tsinghua University, Beijing, 100084, P.R. China \\
E-mail: hli@math.tsinghua.edu.cn
\end{flushleft}

\begin{flushleft}
Zhizhang Wang,\\
{School of Mathematical Sciences, Fudan University, Shanghai, 200433, P.R. China}\\
E-mail: zzwang@fudan.edu.cn	
		\end{flushleft}

\end{document}